\pdfoutput=1    

\documentclass[11pt,a4paper]{article} 

\usepackage[OT2,T1]{fontenc}
\usepackage[utf8]{inputenc}
\usepackage{amsmath}
\usepackage{amssymb}
\usepackage{amsfonts}
\usepackage{amsthm}
\usepackage{pifont}
\usepackage{graphicx}
\usepackage{datetime}
\usepackage{scrdate}
\usepackage{enumitem}
\usepackage{bigints}
\usepackage{longtable}
\usepackage{geometry}
\usepackage[
  backend=biber,
  natbib=true,
  uniquename=false
  ]{biblatex}
\usepackage{csquotes} 
\usepackage{parskip}  
\usepackage{sidecap} 
\sidecaptionvpos{figure}{c} 
\usepackage{hyperref}
\usepackage{etoolbox} 
\apptocmd{\UrlBreaks}{\do\-\do\?}{}{}
\usepackage{chngcntr} 
\usepackage{esvect}
\usepackage{booktabs} 
\usepackage{soul} 
\usepackage{caption}

\hypersetup{colorlinks=true, linkcolor=red, citecolor=red, pdfstartview=FitH, linktocpage=false}

\begingroup
    \makeatletter
    \@for\theoremstyle:=definition,remark,plain\do{%
        \expandafter\g@addto@macro\csname th@\theoremstyle\endcsname{%
            \addtolength\thm@preskip\parskip
            }%
        }
\endgroup

\newtheorem{thm}{Theorem}[section]
\newtheorem{lem}{Lemma}[section]
\newtheorem{coro}{Corollary}[section]

\theoremstyle{definition}
\newtheorem{remark}{Remark}[section]

\numberwithin{equation}{section} 
\counterwithin{figure}{section}
\counterwithin{table}{section}

\makeatletter 
\newcommand\mynobreakpar{\par\nobreak\@afterheading} 
\makeatother

\newcommand{\beq}{\begin{equation*}}
\newcommand{\eeq}{\end{equation*}}
\newcommand{\beqn}{\begin{equation}}
\newcommand{\eeqn}{\end{equation}}
\newcommand{\N}{\mathbb{N}}
\newcommand{\R}{\mathbb{R}}
\newcommand{\C}{\mathbb{C}}
\newcommand{\dd}{\mathrm{d}}
\newcommand{\ee}{\mathrm{e}}
\newcommand{\ii}{\mathrm{i}}
\newcommand{\tn}{\textnormal}
\newcommand{\ds}{\displaystyle}
\newcommand{\ph}{\varphi}
\newcommand{\la}{\lambda}
\newcommand{\rh}{\varrho}
\newcommand{\tr}{$\triangle$}
\newcommand{\tb}{\textbullet}
\DeclareMathOperator{\Rez}{Re}
\DeclareMathOperator{\Imz}{Im}
\DeclareMathOperator{\sgn}{sgn}

\newcommand{\hXi}{\widetilde{h}}
\newcommand{\hf}{h_{\:\!\!f}}
\newcommand{\NMkpm}{N_{M\!,\:\!k,\:\!\pm}}
\newcommand{\NXiMkpm}{\widetilde{N}_{M\!,\:\!k,\:\!\pm}}
\newcommand{\TXi}{\widetilde T}
\newcommand{\TMkpm}{T_{M\!,\:\!k,\:\!\pm}}
\newcommand{\TXiMkpm}{\widetilde{T}_{M\!,\:\!k,\:\!\pm}}
\newcommand{\XBpm}{X_{B\!\!\;,\:\!\pm}}
\newcommand{\XBp}{X_{B\!\!\;,\:\!+}}
\newcommand{\XBm}{X_{B\!\!\;,\:\!-}}
\newcommand{\XFkpm}{X_{F\!,\:\!k,\:\!\pm}}
\newcommand{\XFkp}{X_{F\!,\:\!k,\:\!+}}
\newcommand{\XFkm}{X_{F\!,\:\!k,\:\!-}}
\newcommand{\XFpm}{X_{F\!,\:\!\pm}}
\newcommand{\XFp}{X_{F\!,\:\!+}}
\newcommand{\XFm}{X_{F\!,\:\!-}}
\newcommand{\xFkpm}{x_{F\!,\:\!k,\:\!\pm}}
\newcommand{\XMkpm}{X_{M\!,\:\!k,\:\!\pm}}
\newcommand{\XMkp}{X_{M\!,\:\!k,\:\!+}}
\newcommand{\XMkm}{X_{M\!,\:\!k,\:\!-}}
\newcommand{\Xpdpm}{X_{p,d,\pm}}
\newcommand{\XWkpm}{X_{W\!,\:\!k,\:\!\pm}}
\newcommand{\xWkpm}{x_{W\!,\:\!k,\:\!\pm}}
\newcommand{\XWkp}{X_{W\!,\:\!k,\:\!+}}
\newcommand{\XWkm}{X_{W\!,\:\!k,\:\!-}}
\newcommand{\Xrkpm}{X_{\rh,\:\!k,\:\!\pm}}
\newcommand{\xrkpm}{x_{\rh,\:\!k,\:\!\pm}}
\newcommand{\Xrkp}{X_{\rh,\:\!k,\:\!+}}
\newcommand{\Xrkm}{X_{\rh,\:\!k,\:\!-}}
\newcommand{\yFkpm}{y_{F\!,\:\!k,\:\!\pm}}
\newcommand{\yWkpm}{y_{W\!,\:\!k,\:\!\pm}}
\newcommand{\yrkpm}{y_{\rh,\:\!k,\:\!\pm}}
\newcommand{\etaFkpm}{\eta_{F\!,\:\!k,\:\!\pm}}
\newcommand{\etaWkpm}{\eta_{W\!,\:\!k,\:\!\pm}}
\newcommand{\etarkpm}{\eta_{\rh,\:\!k,\:\!\pm}} 

\newcommand{\kaXi}{\widetilde{\kappa}}
\newcommand{\lak}{\la_{k,\pm}}
\newcommand{\XiBpm}{\varXi_{B\!\!\;,\:\!\pm}}
\newcommand{\XiFpm}{\varXi_{F\!,\:\!\pm}}
\newcommand{\XiFkpm}{\varXi_{F\!,\:\!k,\:\!\pm}}
\newcommand{\XiFkp}{\varXi_{F\!,\:\!k,\:\!+}}
\newcommand{\XiFkm}{\varXi_{F\!,\:\!k,\:\!-}}
\newcommand{\xiFkpm}{\xi_{F\!,\:\!k,\:\!\pm}}
\newcommand{\XiMkpm}{\varXi_{M\!,\:\!k,\:\!\pm}}
\newcommand{\Xipdpm}{\varXi_{p,d,\pm}}
\newcommand{\XiWkpm}{\varXi_{W\!,\:\!k,\:\!\pm}}
\newcommand{\xiWkpm}{\xi_{W\!,\:\!k,\:\!\pm}}
\newcommand{\Xirkpm}{\varXi_{\rh,\:\!k,\:\!\pm}}
\newcommand{\Xirkp}{\varXi_{\rh,\:\!k,\:\!+}}
\newcommand{\Xirkm}{\varXi_{\rh,\:\!k,\:\!-}}
\newcommand{\xirkpm}{\xi_{\rh,\:\!k,\:\!\pm}}
\newcommand{\phAkpE}{\ph_{A,\:\!k,\:\!+,\:\!1}}
\newcommand{\phAkpZ}{\ph_{A,\:\!k,\:\!+,\:\!2}}
\newcommand{\phAkmE}{\ph_{A,\:\!k,\:\!-,\:\!1}}
\newcommand{\phAkmZ}{\ph_{A,\:\!k,\:\!-,\:\!2}}
\newcommand{\phBkpm}{\ph_{B,\:\!k,\:\!\pm}}
\newcommand{\phBkp}{\ph_{B,\:\!k,\:\!+}}
\newcommand{\phBkm}{\ph_{B,\:\!k,\:\!-}}
\newcommand{\phFkpE}{\ph_{F\!,\:\!k,\:\!+,\:\!1}}
\newcommand{\phFkpZ}{\ph_{F\!,\:\!k,\:\!+,\:\!2}}
\newcommand{\phFkmE}{\ph_{F\!,\:\!k,\:\!-,\:\!1}}
\newcommand{\phFkmZ}{\ph_{F\!,\:\!k,\:\!-,\:\!2}}
\newcommand{\phpkpE}{\ph_{p,\:\!k,\:\!+,\:\!1}}
\newcommand{\phpkpZ}{\ph_{p,\:\!k,\:\!+,\:\!2}}
\newcommand{\phSkpm}{\ph_{S,\:\!k,\:\!\pm}}
\newcommand{\phSkp}{\ph_{S,\:\!k,\:\!+}}
\newcommand{\phSkm}{\ph_{S,\:\!k,\:\!-}}
\newcommand{\phSpm}{\ph_{S,\pm}}
\newcommand{\phSp}{\ph_{S,+}}
\newcommand{\phSm}{\ph_{S,-}}
\newcommand{\phXiAkpE}{\widetilde{\ph}_{A,\:\!k,\:\!+,\:\!1}}
\newcommand{\phXiAkpZ}{\widetilde{\ph}_{A,\:\!k,\:\!+,\:\!2}}
\newcommand{\phXiAkmE}{\widetilde{\ph}_{A,\:\!k,\:\!-,\:\!1}}
\newcommand{\phXiAkmZ}{\widetilde{\ph}_{A,\:\!k,\:\!-,\:\!2}}
\newcommand{\phXiBkpm}{\widetilde{\ph}_{B,\:\!k,\:\!\pm}}
\newcommand{\phXiBkp}{\widetilde{\ph}_{B,\:\!k,\:\!+}}
\newcommand{\phXiBkm}{\widetilde{\ph}_{B,\:\!k,\:\!-}}
\newcommand{\phXiFkpE}{\widetilde{\ph}_{F,\:\!k,\:\!+,\:\!1}}
\newcommand{\phXiFkpZ}{\widetilde{\ph}_{F,\:\!k,\:\!+,\:\!2}}
\newcommand{\phXiFkmE}{\widetilde{\ph}_{F,\:\!k,\:\!-,\:\!1}}
\newcommand{\phXiFkmZ}{\widetilde{\ph}_{F,\:\!k,\:\!-,\:\!2}}
\newcommand{\phXipkpE}{\widetilde{\ph}_{p,\:\!k,\:\!+,\:\!1}}
\newcommand{\phXipkpZ}{\widetilde{\ph}_{p,\:\!k,\:\!+,\:\!2}}
\newcommand{\phXiSkpm}{\widetilde{\ph}_{S,\:\!k,\:\!\pm}}

\begin{document}

\title{Determining the geometry of noncircular gears\\ for given transmission function} 
\author{Uwe B\"asel}
\date{} 
\maketitle
\thispagestyle{empty}
\vspace{-0.3cm}
\begin{abstract}
\noindent A pair of noncircular gears can be used to generate a strictly increasing continuous function $\psi(\varphi)$ whose derivative $\psi'(\varphi) = \mathrm{d}\psi(\varphi)/\mathrm{d}\varphi > 0$ is $2\pi/n$-periodic, where $\varphi$ and $\gamma = \psi(\varphi)$ are the angles of the opposite rotation directions of the drive gear and the driven gear, respectively, and $n \in \mathbb{N} \setminus \{0\}$.
In this paper, we determine the geometry of both gears for given transmission function $\psi(\varphi)$ when manufacturing with a rack-cutter having fillets.
All occurring functions are consistently derived as functions of the drive angle $\varphi$ and the function $\psi$. 
Throughout the paper, methods of complex algebra, including an external product, are used.
An effective algorithm for the calculation of the tooth geometries -- in general every tooth has its own shape -- is presented which limits the required numerical integrations to a minimum.
Simple criteria are developed for checking each tooth flank for undercut.
The base curves of both gears are derived, and it is shown that the tooth flanks are indeed the involutes of the corresponding base curve.  
All formulas for both gears are ready to use.
\\[0.2cm]
\textbf{2010 Mathematics Subject Classification:}
53A04, 
53A17, 
51N20, 
15A75  
\\[0.2cm]
\textbf{Keywords:} Noncircular gears, involute gears, transmission function, envelope, undercut, evolute, exterior product
\end{abstract}

\section{Introduction}

Toothed gears are used in many technical systems.
The most important and most frequently used type of gearing nowadays is the involute gearing dating back to Leonhard Euler (1765/67)~\cite{Euler:Supplementum}.
Involute gears have the advantages that they are easy to manufacture -- hence cost-effective -- and ensure a constant transmission ratio even if the centre distance is not exactly maintained.
Furthermore, they have good mating properties.
The working flanks of the teeth are the involutes of a circle, which is called base circle.
(Clearly, the base circle is the evolute of these involutes.)
For the manufacture of involute gears, straight flank tools are often used, for example so-called rack-cutters.
(The profile of such a rack-cutter is shown in Fig.\ \ref{Fig:Reference_profile01}.) 

In addition to the well-known circular gears for constant transmission ratios, there are also noncircular gears for non-constant transmission ratios, which also play an important role (see \textcite[Chapter 12]{Litvin&Fuentes}, \textcite{litvin2014noncircular}).
They are often used for the generation of motions with periodic transmission ratio, such as those required in printing, packaging, textile and many other types of machines.
Elliptical gears  (see e.g.\ \textcite[pp.\ 233-235, 242-244]{Wunderlich}, \textcite{Chen&Tsay}, \textcite[pp.\ 324-327]{Litvin&Fuentes},  \textcite{Litvin&Gonzalez-Perez&Yukishima&Fuentes&Hayasaka}, \cite[pp.\ 40-52]{litvin2014noncircular}), for example, can be applied for this purpose.
However, these only provide suitable rotational motions to a limited extent (see Fig.\ 172 in \cite[p.\ 235]{Wunderlich}).
It is better to specify the required motion (transmission function) and dimension the gear drive so that it realizes this motion.
\textcite{Tsay&Fong} use Fourier series approximation of centrodes and gear ratios in order to derive the tooth profiles of noncircular gears (see e.g. Example 3 on pp.\ 567-568).
The determination of the centrodes for given transmission functions is investigated in \cite{Litvin&Gonzalez-Perez&Fuentes&Hayasaka08}, especially in Section 4.
\textcite{Qiu&Deng} calculate the tooth profile for given transmission function when manufacturing with rack-cutter, but only for the left sides of the gear tooth spaces.
The calculations require case distinctions that are not discussed.

In this paper, we calculate the complete tooth profiles of the drive and the driven gear when manufacturing with a rack-cutter having fillets. All calculations are based on the transmission function, and
so all functions are consistently related to the rotation angle $\ph$ of the drive gear.
We use complex-valued functions of the real variable $\ph$ (see e.g.\ \textcite{Wunderlich}, \textcite{Mueller:Kinematik}, \textcite{Luck&Modler}, \textcite{Laczik&Zentay&Horvath}).
This allows a very simple and transparent derivation of the results.
Undercut conditions are obtained for both gears.
Equations for the base curves are derived.
A detailed algorithm summarizes all necessary calculation steps.


The paper is organized as follows: \mynobreakpar
\begin{itemize}[leftmargin=0.5cm]
\setlength{\itemsep}{-2pt}
\item In Section \ref{Sec:Preliminaries} we introduce some basic notation and foundations that we will use throughout the paper.
\item In Section \ref{Sec:Drive_gear} we derive a parametric equation for the flank curves of the drive gear in two different ways: 1) by obtaining the flank curves as envelopes of the rack-cutter flanks, applying an exterior product for complex numbers, 2) using the instantaneous centre of velocity of drive gear and rack-cutter.
Furthermore, we derive a parametric equation for the fillet curves.
\item Section \ref{Sec:Driven_gear} provides the calculations for the driven gear analogous to those of Section \ref{Sec:Drive_gear}. 
\item In Section \ref{Sec:Undercut}, undercut conditions are derived for the drive and the driven gear.
\item Parametric equations for the base curves are derived in Section \ref{Sec:The_base_curves}.
As a by-product the curvatures of the flank curves are obtained. 
\item Section \ref{Sec:Algorithm} is a detailed algorithm for the complete profiles of both gears.
\item Section \ref{Sec:Example} gives a practical example.
\item Calculation rules for the exterior product are to be found in Appendix \ref{App:Exterior_product}, a list of the used formula symbols in Appendix \ref{App:Formula_symbols}.
\end{itemize} 

\section{Preliminaries}
\label{Sec:Preliminaries}

We denote by $\ph$ the rotation angle of the drive gear, and by $\gamma$ the rotation angle of the driven gear.
Let
\beqn \label{Eq:psi}
  \psi \colon [0,2\pi] \;\rightarrow\; \R\,,\quad
  \ph \;\mapsto\; \psi(\ph)
\eeqn
be a strictly increasing continuous function whose derivative
\beqn \label{Eq:psi'}
  \psi'(\ph)
= \frac{\dd \psi(\ph)}{\dd \ph} > 0
\eeqn
is $2\pi/n$-periodic, $n \in \N \setminus \{0\}$.
Our aim is to determine drive and driven gear in such a way that the relationship between $\ph$ and $\gamma$ is given by
\beqn \label{Eq:gamma}
  \gamma = \psi(\ph)\,.
\eeqn
The function $\psi$ is called {\em transmission function}.

Let $\ph^*(t)$ denote the rotation angle of the drive gear as function of time $t$.
Then, the rotation angle of the driven gear as time function is given by
\beqn \label{Eq:psi^*}
  \psi^*(t) := \psi(\ph^*(t))\,.
\eeqn
With $\ph = \ph^*(t)$ we get
\beqn \label{Eq:dpsi^*/dt}
  \frac{\dd\psi^*(t)}{\dd t}
= \frac{\dd\psi(\ph^*(t))}{\dd t}
= \frac{\dd\psi}{\dd\ph}\,(\ph^*(t)) \cdot \frac{\dd\ph^*(t)}{\dd t}  
\eeqn
From \eqref{Eq:dpsi^*/dt} with \eqref{Eq:psi'} it follows that 
\beqn \label{Eq:psi'(phi(t))}
  \psi'(\ph^*(t))
= \frac{\dd\psi}{\dd\ph}\,(\ph^*(t))
= \frac{{\dot\psi}^*(t)}{{\dot\ph}^*(t)}\,,  
\eeqn
where the over-dot indicates time derivatives.
So one sees that $\psi'$ is the ratio of the angular velocities $\dot{\psi}^*$ and $\dot{\ph}^*$ of driven gear and drive gear, respectively.

Let $P$ be the instanteneous center of relative rotation of drive and driven gear (see Fig.\ \ref{Fig:Waelzkurven01}).
With the distance $a = \left|\overline{O_1 O_2}\right|$ between the pivot points $O_1$ and $O_2$ we get 
\beq
  r(\ph) + R(\ph) = a
  \qquad\mbox{and}\qquad
  \psi'(\ph)
= \frac{\left|\overline{O_1 P}\right|}{\left|\overline{O_2 P}\right|}
= \frac{r(\ph)}{R(\ph)}\,.  
\eeq
It follows that
\beqn \label{Eq:r_and_R}
  r(\ph)
= \frac{a \psi'(\ph)}{\ds 1 + \psi'(\ph)}
  \qquad\mbox{and}\qquad
  R(\ph)
= \frac{a}{\ds 1 + \psi'(\ph)}\,.  
\eeqn
The instantaneous center $P$ is in the rotating $x,y$-system given by $X_P(\ph) = r(\ph)\,\ee^{-\ii\ph}$, hence here the path of $P$ is the curve (centrode of the drive gear)
\beqn \label{Eq:X_P}
  X_P \colon [0,2\pi] \;\rightarrow\; \C\,,\quad
  \ph \;\mapsto\; X_P(\ph) = r(\ph)\,\ee^{-\ii\ph}\,.
\eeqn
In the rotating $\xi,\eta$-system, the path of $P$ is the curve (centrode of the driven gear)
\beqn \label{Eq:Xi_P}
  \varXi_P \colon [0,2\pi]  \;\rightarrow\; \C\,,\quad
  \ph \;\mapsto\; \varXi_P(\ph) = -R(\ph)\,\ee^{\ii \psi(\ph)}\,.
\eeqn
\begin{SCfigure}[][h]
  \includegraphics[scale=1]{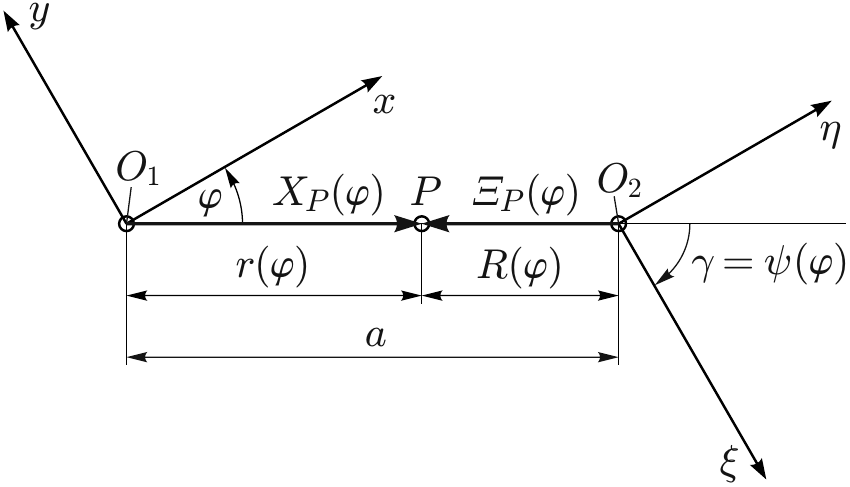}
  \caption{Generation of the centrodes $X_P$ and $\varXi_P$}
  \label{Fig:Waelzkurven01}
\end{SCfigure}

The arc length of $X_P$ between the points $X_P(\ph_0)$ and $X_P(\ph_1)$ is given by
\beqn \label{Eq:s}
  s(\ph_0,\ph_1)
= \int_{\ph_0}^{\ph_1} |X_P'(\ph)|\,\dd\ph\,,
\eeqn
where $X_P'(\ph) = \frac{\dd X_P(\ph)}{\dd\ph}$ is the tangent vector of $X_P$ at point $X_P(\ph)$ (momentary contact point of $X_P$ and $\varXi_P$).
From \eqref{Eq:X_P} one gets
\beq
  X_P'(\ph)
= r'(\ph)\ee^{-\ii\ph} - r(\ph)\ii\ee^{-\ii\ph}
= \left(r'(\ph) - \ii r(\ph)\right)\ee^{-\ii\ph}\,,  
\eeq
hence
\begin{align*}
  |X_P'(\ph)|
= {} & \left|\:\!\!\left(r'(\ph) - \ii r(\ph)\right)\ee^{-\ii\ph}\right|
= \left|r'(\ph) - \ii r(\ph)\right| \left|\ee^{-\ii\ph}\right|\\[0.1cm]
= {} & \left|r'(\ph) - \ii r(\ph)\right|
= \sqrt{r'^2(\ph) + r^2(\ph)}\,.
\end{align*}
From $r(\ph)$ in \eqref{Eq:r_and_R} we find
\beq
  r'(\ph)
= \frac{a \psi''(\ph)}{\ds (1 + \psi'(\ph))^2}\,.  
\eeq
It follows that
\beq 
  X_P'(\ph)
= \left(\frac{a \psi''(\ph)}{\ds (1 + \psi'(\ph))^2}
  - \frac{a \ii \psi'(\ph)}{\ds 1 + \psi'(\ph)}\right) \ee^{-\ii\ph}
= \frac{a \psi''(\ph) - a \ii \psi'(\ph)(1 + \psi'(\ph))}{\ds (1 + \psi'(\ph))^2}\,\ee^{-\ii\ph}   
\eeq
and
\beqn \label{Eq:|X_P'|-1}
  |X_P'(\ph)|
= \sqrt{\frac{a^2 \psi''^2(\ph)}{(1 + \psi'(\ph))^4} + \frac{a^2 \psi'^2(\ph)}{(1 + \psi'(\ph))^2}}
= \frac{a w(\ph)}{(1 + \psi'(\ph))^2}
\eeqn
with
\beqn \label{Eq:w}
  w(\ph)
:= \sqrt{\strut\ds \psi''^2(\ph) + \psi'^2(\ph)\left(1 + \psi'(\ph)\right)^2}\,.   
\eeqn
So we have found
\beq
  s(\ph_0,\ph_1)
= a I(\ph_0,\ph_1)\,,
\eeq
where
\beqn \label{Eq:I}
  I(\ph_0,\ph_1)
:= \int_{\ph_0}^{\ph_1} \frac{w(\ph)}{(1 + \psi'(\ph))^2}\,\dd\ph\,.  
\eeqn
The perimeter $u$ of $X_P$ ist given by $u = aI(0,2\pi)$.
We have
\beq
  u = aI(0,2\pi) = z_1 p = z_1 \pi m\,, 
\eeq
where $z_1$ is the number of teeth of the drive gear, $p$ the tooth pitch and $m$ the module (see Fig.\ \ref{Fig:Reference_profile01}), hence
\beqn \label{Eq:a}
  a
= \frac{z_1 \pi m}{I(0,2\pi)}\,.  
\eeqn

\begin{SCfigure}[][h]
  \includegraphics[scale=1]{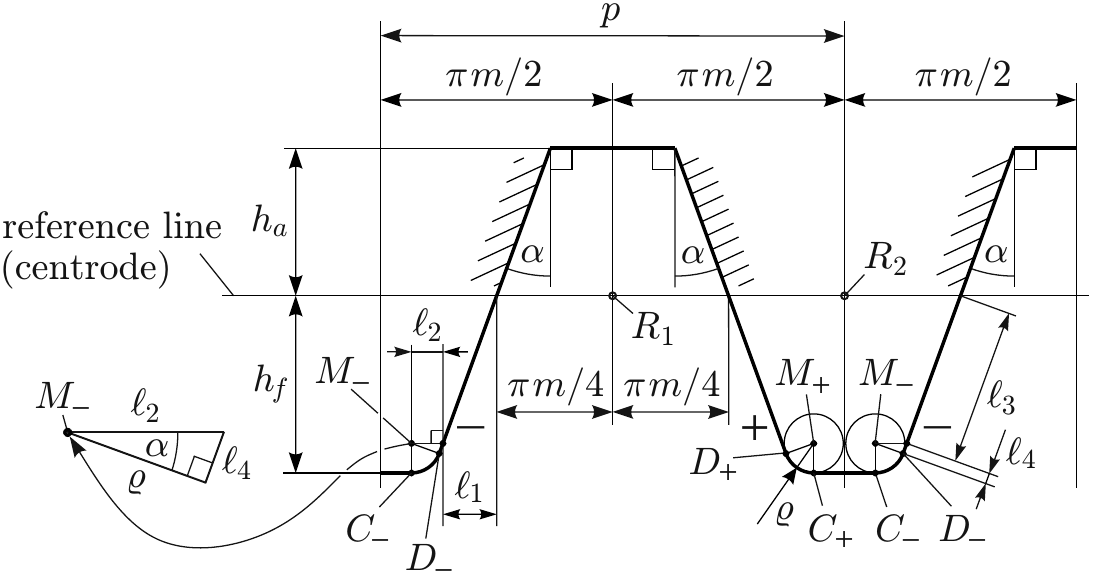}
  \caption{Reference profile of the rack-cutter}
  \label{Fig:Reference_profile01}
\end{SCfigure}

Fig.\ \ref{Fig:Definition_flanks01} is crucial for all further considerations and calculations.
The reference line (centrode) of the rack-cutter is always tangent to the centrodes $X_P$ and $\varXi_P$ at the intanteneous rotation center $P$.
During the rotations of $X_P$ and $\varXi_P$ about $O_1$ and $O_2$, respectively, the reference line is shifted between $X_P$ and $\varXi_P$, where it generally rotates around $P$. 
Rigidly connected to the reference line are the two flanks marked with ``$+$'' and ``$-$''.
(Note that $P$ in general is not in the middle between both flanks.)
The relative motion between one of the centrodes and the reference line is a pure rolling.
The flanks ``$+$'' and ``$-$'' produce the corresponding flank curves ``$+$'' and ``$-$'', respectively, of \underline{tooth $k$ of the drive gear} ($\XFkpm$ in Fig.\ \ref{Fig:Curves_for_tooth_k}) and \underline{tooth space $k$ of the driven gear}.

{\bf Convention:} In order to abbreviate the text, the following convention will be used throughout the paper.
In formula symbols and formulas that contain ``$\pm$'' and sometimes also ``$\mp$'', the upper sign is for the ``$+$''-flank and the lower sign for the ``$-$''-flank.
Example: The sentences ``The point $M_\pm$ in Fig.\ \ref{Fig:Reference_profile01} generates the curve $\XMkpm$ in Fig.\ \ref{Fig:Curves_for_tooth_k}.'' and ``The points $M_\pm$ in Fig.\ \ref{Fig:Reference_profile01} generate the curves $\XMkpm$ in Fig.\ \ref{Fig:Curves_for_tooth_k}.'' are abbreviations for the sentence ``The points $M_+$ and $M_-$ in Fig.\ \ref{Fig:Reference_profile01} generate the curves $\XMkp$ and $\XMkm$, respectively, in Fig.\ \ref{Fig:Curves_for_tooth_k}.'' (Since Fig.\ \ref{Fig:Curves_for_tooth_k} shows the curves for a tooth of the drive gear, the point $M_-$ is the left one in Fig.\ \ref{Fig:Reference_profile01}.)           

\begin{SCfigure}[][h]
  \includegraphics[scale=1]{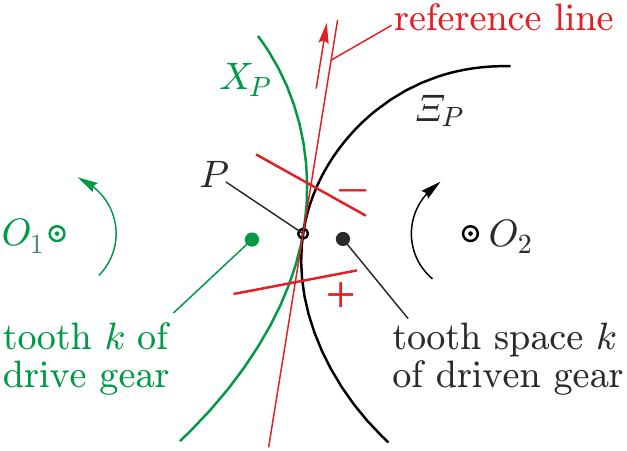}
  \caption{Definition of the flanks ``$+$'' and~``$-$''}
  \label{Fig:Definition_flanks01}
\end{SCfigure}

\section{Drive gear}
\label{Sec:Drive_gear}

The tangent unit vector $T$ at point $X_P(\ph)$ of $X_P$ (see Fig.\ \ref{Fig:Tooth_generation01a}) is given by
\beqn \label{Eq:T}
  T(\ph)
:= \frac{X_P'(\ph)}{|X_P'(\ph)|}
= \frac{\psi''(\ph) - \ii \psi'(\ph)(1 + \psi'(\ph))}{w(\ph)}\,\ee^{-\ii\ph}\,.
\eeqn
Splitting \eqref{Eq:T} into real and imaginary part gives
\beqn \label{Eq:tx_and_ty}
\left.
\begin{aligned}
  t_x(\ph) := \Rez T(\ph)   
= {} & {\phantom{-}\frac{\psi''(\ph)\cos{\ph} - \psi'(\ph)(1 + \psi'(\ph))\sin\ph}{w(\ph)}}\,,\\[0.1cm]
  t_y(\ph) := \Imz T(\ph)
= {} & {-\frac{\psi''(\ph)\sin{\ph} + \psi'(\ph)(1 + \psi'(\ph))\cos\ph}{w(\ph)}}\,.      
\end{aligned}
\;\right\}
\eeqn
We denote the value of $\ph$ in the middle of the tooth $k$ to be generated by $\chi(k)$, $k = 1,\ldots,z_1$.
One gets these values -- in general numerically -- from
\beqn \label{Eq:chi(k)}
  a I(0,\chi(k)) = (k-1) \pi m
\eeqn
with $I(\cdot\,,\cdot)$ according to \eqref{Eq:I}, and $a$ according to \eqref{Eq:a}.

\begin{SCfigure}[][h]
  \includegraphics[scale=1]{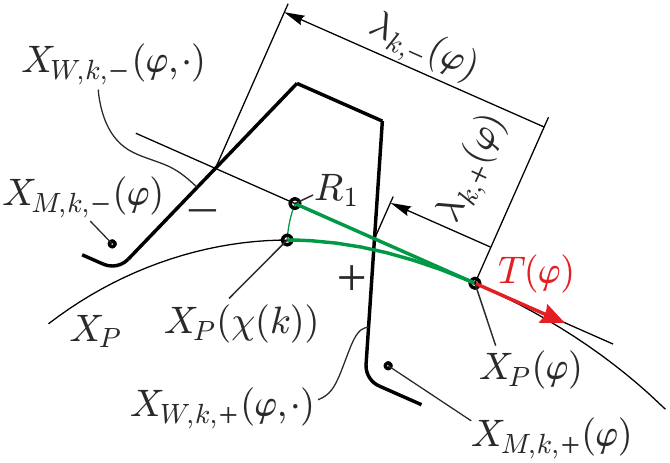}
  \caption{Motion of the rack-cutter during the generation of tooth $k$}
  \label{Fig:Tooth_generation01a}   
\end{SCfigure}

Now we determine a parametric equation for the rack-cutter flank line $\XWkpm$ (see Fig.\ \ref{Fig:Tooth_generation01a}) that generates tooth flank curve $\XFkpm$, $k = 1,2,\ldots,z_1$ (see Fig.\ \ref{Fig:Curves_for_tooth_k}). 
If $\ph = \chi(k)$, then the point $R_1$ of the rack-cutter coincides with the point $X_P(\chi(k))$ of the centrode $X_P$ (see Fig.\ \ref{Fig:Tooth_generation01a}).
The vectors $T(\ph)$ and $\vv{X_P(\ph)R_1}$ point into the same direction if $I(\chi(k),\ph) < 0$, which is not the case in Fig.\ \ref{Fig:Tooth_generation01a}.
The signed arc length between the points $X_P(\chi(k))$ and $X_P(\ph)$ ($=$ signed distance between $R_1$ and $X_P(\ph)$) is equal to $aI(\chi(k),\ph)$.
The signed distances $\la_{k,+}(\ph)$ and $\la_{k,-}(\ph)$ are given by 
\beqn \label{Eq:lambda_k}
  \lak(\ph)
:= \pm\frac{\pi m}{4} - a I(\chi(k),\ph)\,.   
\eeqn
$\lak(\ph) > 0$ indicates that $T(\ph)$ and the arrow of $\lak(\ph)$ have the same direction.   
So we have the following parametric equation for the rack-cutter flank line $\XWkpm$
\beqn \label{Eq:XWkpm}
\begin{aligned} 
  \XWkpm(\ph,\mu)
= {} & X_P(\ph) + \lak(\ph)\,T(\ph)
  + \mu\,T(\ph)\,\ee^{\ii\left(\frac{\pi}{2}\,\pm\,\alpha\right)}\\[0.05cm]
= {} & X_P(\ph) + \lak(\ph)\,T(\ph)
  + \mu\,T(\ph)\,\ii\ee^{\pm\ii\alpha}\\[0.05cm]
= {} & X_P(\ph) + \left(\lak(\ph)+\mu\ii\ee^{\pm\ii\alpha}\right) T(\ph)\,,
  \quad \mu \in \R\,.    
\end{aligned}
\eeqn
Splitting \eqref{Eq:XWkpm} into real and imaginary part gives
\beqn \label{Eq:xWk_and_yWk}
\left.
\begin{aligned}
  \xWkpm(\ph,\mu)
= {} & \hspace{0.26cm}r(\ph)\cos(\ph) + \lak(\ph)\,t_x(\ph)
  + \mu\,(\mp t_x(\ph)\sin\alpha - t_y(\ph)\cos\alpha)\,,\\[0.1cm]
  \yWkpm(\ph,\mu)
= {} & {-r(\ph)\sin(\ph)} + \lak(\ph)\,t_y(\ph)
  + \mu\,(\mp t_y(\ph)\sin\alpha + t_x(\ph)\cos\alpha)\,.   
\end{aligned}
\;\right\}
\eeqn
Now, we determine a parametric equation of the tooth flank curves $\XFkpm$.

\begin{SCfigure}[][h]
  \includegraphics[scale=1]{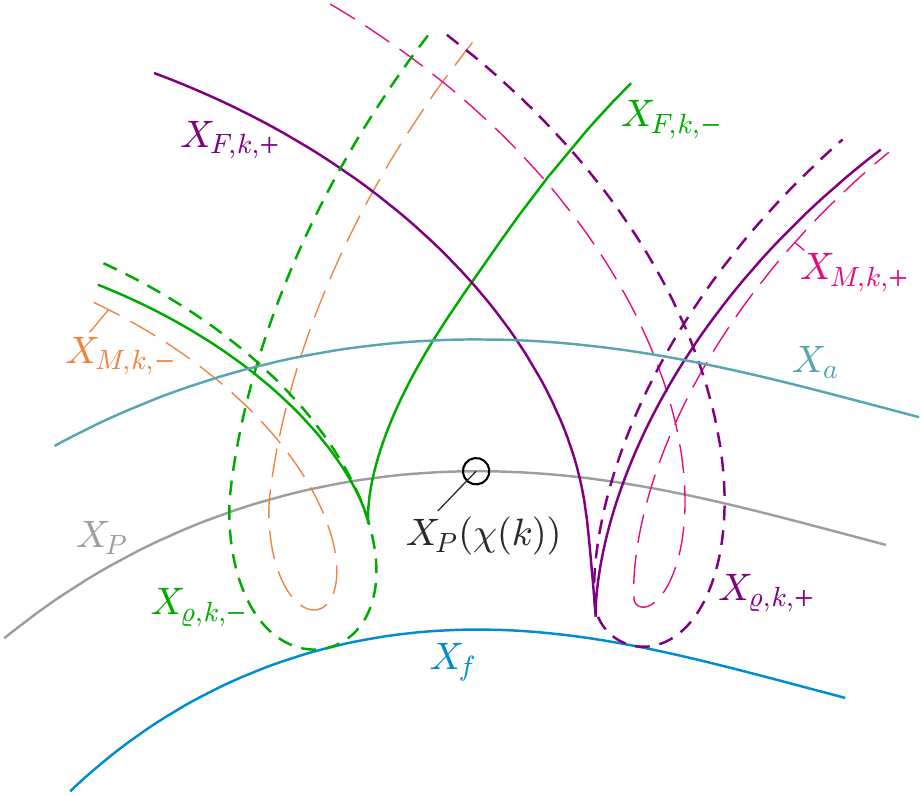}
  \caption{Curves for tooth $k$ (example):
    \vspace*{-0.3cm}
	\newline $\XFkpm$ flank curves,
	\newline $\XMkpm$ curves of $M_\pm$
	\newline\hspace*{1.15cm} (see Fig.\ \ref{Fig:Reference_profile01}),
	\newline $\Xrkpm$ fillet curves,
	\newline $X_P$ centrode,
	\newline $X_a$ addendum curve,
	\newline $X_f$ dedendum curve}
  \label{Fig:Curves_for_tooth_k}   
\end{SCfigure}

\begin{thm} \label{Thm:X_F}
A parametric equation of the flank curves $\XFkpm$ {\em(}see Fig.\ {\em \ref{Fig:Curves_for_tooth_k})} of tooth $k$, $k = 1,2,\ldots,z_1$, of the drive gear is given by
\beqn \label{Eq:XFkpm}
  \XFkpm(\ph)
= X_P(\ph) + \lak(\ph)\,T(\ph)\,\ee^{\pm\ii\alpha}\cos\alpha\,,
\eeqn
where
\begin{gather*}
  X_P(\ph)
= r(\ph)\,\ee^{-\ii\ph}\,,\qquad 
  T(\ph)
= \frac{\psi''(\ph) - \ii \psi'(\ph)(1 + \psi'(\ph))}{w(\ph)}\,\ee^{-\ii\ph}\,,\\[0.1cm]
  \lak(\ph)
= \pm\frac{\pi m}{4} - a \int_{\chi(k)}^\ph \frac{w(\phi)}{\ds (1 + \psi'(\phi))^2}\,\dd\phi\,.   
\end{gather*}
with
\beq
  r(\ph)
= \frac{a \psi'(\ph)}{\ds 1 + \psi'(\ph)}\,,\qquad
  w(\ph)
= \sqrt{\strut\ds \psi''^2(\ph) + \psi'^2(\ph)\left(1 + \psi'(\ph)\right)^2}\,,
\eeq
and $\chi(k)$ according to \eqref{Eq:chi(k)}.
\end{thm}

\begin{proof}
Choosing two points (complex numbers) $A = A(\ph)$, $B = B(\ph)$ of the rack-cutter flank line $\XWkp(\ph,\cdot)$ or $\XWkm(\ph,\cdot)$, the equation of this line can be written as
\beqn \label{Eq:Family}
  [A(\ph)-B(\ph),X] = [A(\ph),B(\ph)]
\eeqn
(cp.\ \eqref{Eq:EOL}).
A tooth flank curve is the envelope (see e.g.\ \textcite[pp.\ 94-98]{Baule1}) of the family \eqref{Eq:Family} of cutter flank lines with $\ph$ varying in a suitable interval.
The equation of the envelope can be obtained from the linear equation system 
\beq
\left.
\begin{array}{r@{\;=\;}l}
  [A-B,X] & [A,B]\,,\\[0.1cm]
  [(A-B)',X] & [A,B]'
\end{array}  
\right\}
\eeq
which, according to \eqref{Eq:[A,B]-Def}, can be written as 
\beqn \label{Eq:CES}
\left.
\begin{array}{l@{\,}l@{\:-\:}l@{\,}c@{\;=\;}l}
  \overline{(A-B)}  & X & (A-B)  & \overline{X}  & 2\ii[A,B]\:,\\[0.1cm] 
  \overline{(A-B)'} & X & (A-B)' & \overline{X} & 2\ii[A,B]'.
\end{array}
\right\}
\eeqn
The solution of \eqref{Eq:CES} is given by
\begin{align} \label{Eq:X-1}
  X 
= {} & 
  \frac
	{\left|\!\begin{array}{ll}
	  2\ii[A,B]  & -(A-B)\\
	  2\ii[A,B]' & -(A-B)' 	
	 \end{array}\!\right|}
	{\left|\!\begin{array}{ll}
	  \,\overline{(A-B)}  & -(A-B)\\
	  \,\overline{(A-B)'} & -(A-B)'
	 \end{array}\!\right|}
= 2\ii\,
  \frac
	{[A,B]\,(A-B)' - [A,B]'\,(A-B)}
	{\,\overline{(A-B)}\,(A-B)' - (A-B)\,\overline{(A-B)'}\,}\nonumber\\[0.2cm]
= {} &
  \frac
	{[A,B]\,(A-B)' - [A,B]'\,(A-B)}
	{[A-B,(A-B)']}\,.		 	   	    
\end{align}
Let us choose
\begin{align*}
  A = A(\ph)
:= {} & \XWkpm(\ph,1)
= X_P(\ph) + \lak(\ph)\,T(\ph) + \ii\,\ee^{\pm\ii\alpha}\,T(\ph)\,,\\[0.05cm] 
  B = B(\ph)
:= {} & \XWkpm(\ph,0)
= X_P(\ph) + \lak(\ph)\,T(\ph)\,,
\end{align*}
then we have
\beq
  A - B = \ii\,\ee^{\pm\ii\alpha}\,T(\ph)\,,\qquad
  (A - B)' = \ii\,\ee^{\pm\ii\alpha}\,T'(\ph)\,.
\eeq
From \eqref{Eq:Rule3}, it follows that
\begin{align*}
  [A-B,(A-B)']
= {} & [\ii\,\ee^{\pm\ii\alpha}\,T(\ph),\ii\,\ee^{\pm\ii\alpha}\,T'(\ph)]
= \ii\,\ee^{\pm\ii\alpha}\,\overline{\ii\,\ee^{\pm\ii\alpha}}\: [T(\ph),T'(\ph)]\\[0.05cm]
= {} & \ee^{\pm\ii\alpha}\,\ee^{\mp\ii\alpha}\: [T(\ph),T'(\ph)]
= [T(\ph),T'(\ph)]\,. 
\end{align*}
In the following we respectively write $X_P$, $T$, $\la$, $w$ and $\psi$ instead of $X_P(\ph)$, $T(\ph)$, $\lak(\ph)$, $w(\ph)$ and $\psi(\ph)$. 
Applying \eqref{Eq:Rules1}, \eqref{Eq:Rule1}, \eqref{Eq:Rule2} and \eqref{Eq:Rule3}, we obtain
\begin{align*} 
  [A,B]
= {} & [X_P + \la\,T + \ii\,\ee^{\pm\ii\alpha}\,T,\, X_P + \la\,T]\nonumber\\[0.05cm]  
= {} & [X_P + \la\,T,\, X_P + \la\,T] + [\ii\,\ee^{\pm\ii\alpha}\,T,\, X_P + \la\,T]\nonumber\\[0.05cm]
= {} & [\ii\,\ee^{\pm\ii\alpha}\,T,\, X_P] + \la\,[\ii\,\ee^{\pm\ii\alpha}\,T,\, T]
= [\ii\,\ee^{\pm\ii\alpha}\,T,\, X_P] + \la\,T\,\overline{T\,}[\ii\,\ee^{\pm\ii\alpha},\, 1]\nonumber\\[0.05cm]
= {} & [\ii\,\ee^{\pm\ii\alpha}\,T,\, X_P] - \la\,[1,\,\ii\,\ee^{\pm\ii\alpha}]
= [\ii\,\ee^{\pm\ii\alpha}\,T,\, X_P] - \la\Imz(\ii\,\ee^{\pm\ii\alpha})\nonumber\\[0.05cm]
= {} & [\ii\,\ee^{\pm\ii\alpha}\,T,\, X_P] - \la\cos\alpha\,,
\end{align*}
hence, with \eqref{Eq:T}, \eqref{Eq:Rules1}, \eqref{Eq:Rule1}, \eqref{Eq:Rule2}, \eqref{Eq:Rule3} and \eqref{Eq:[A,B]'},
\begin{align} \label{Eq:[A,B]'-1}
  [A,B]'
= {} & [\ii\,\ee^{\pm\ii\alpha}\,T',\, X_P]
  + [\ii\,\ee^{\pm\ii\alpha}\,T,\, X_P'] - \la'\cos\alpha\nonumber\displaybreak[0]\\[0.05cm]
= {} & [\ii\,\ee^{\pm\ii\alpha}\,T',\, X_P]
  + [\ii\,\ee^{\pm\ii\alpha}\,T,\, |X_P'|\,T] - \la'\cos\alpha\nonumber\displaybreak[0]\\[0.05cm]
= {} & [\ii\,\ee^{\pm\ii\alpha}\,T',\, X_P]
  + |X_P'|\,T\,\overline{T}\,[\ii\,\ee^{\pm\ii\alpha},\, 1] - \la'\cos\alpha\nonumber\\[0.05cm]    
= {} & [\ii\,\ee^{\pm\ii\alpha}\,T',\, X_P]
  - |X_P'|\,[1,\,\ii\,\ee^{\pm\ii\alpha}] - \la'\cos\alpha\nonumber\\[0.05cm]
= {} & [\ii\,\ee^{\pm\ii\alpha}\,T',\, X_P]
  - |X_P'|\cos\alpha - \la'\cos\alpha\,.
\end{align}
For the derivative of $\lak(\ph)$ from \eqref{Eq:lambda_k} we find 
\beqn \label{Eq:lambda_k'}
  \lak'(\ph)
= \frac{\dd}{\dd\ph}\left(\pm\frac{\pi m}{4}
  - a \int_{\chi(k)}^\ph \frac{w(\phi)}{(1 + \psi'(\phi))^2}\,\dd\phi
  \right)
= -\frac{aw(\ph)}{(1 + \psi'(\ph))^2}\,.    
\eeqn
Comparing \eqref{Eq:lambda_k'} with \eqref{Eq:|X_P'|-1} one sees that
\beqn \label{Eq:lambda_k'=-|X_P'|}
  \lak'(\ph)
= -|X_P'(\ph)|  
\eeqn
which simplifies \eqref{Eq:[A,B]'-1} to
\beq
  [A,B]'
= [\ii\,\ee^{\pm\ii\alpha}\,T',\, X_P]\,.
\eeq
Now, \eqref{Eq:X-1} may be written as
\beqn \label{Eq:X-2}
  X
= \frac
	{([\ii\,\ee^{\pm\ii\alpha}\,T,\, X_P] - \la\cos\alpha)\,T' - [\ii\,\ee^{\pm\ii\alpha}\,T',\, X_P]\,T}
	{\tn{$[T,T']$}}
  \,\ii\,\ee^{\pm\ii\alpha}\,. 
\eeqn
It is necessary to determine the derivative of $T(\ph)$.
We write \eqref{Eq:T} as
\beq
  T w \ee^{\ii\ph} = \psi'' - \ii \psi' - \ii \psi'^2\,,
\eeq
and get
\beq
  T' w \ee^{\ii\ph} + T w' \ee^{\ii\ph} + T w \ii \ee^{\ii\ph}
= \psi^{(3)} - \ii \psi'' - 2 \ii \psi' \psi''\,,  
\eeq
hence
\begin{align} \label{Eq:T'-1}
  T'
= {} & \frac{\left(\psi^{(3)} - \ii\psi'' - 2\ii\psi'\psi''\right)\ee^{-\ii\ph} - Tw' - \ii Tw}{w}
  \nonumber\displaybreak[0]\\[0.05cm]  
= {} & \frac{T}{w^2}\,
  \bigg(w\,\frac{\psi^{(3)} - \ii\psi'' - 2\ii\psi'\psi''}{T}\,\ee^{-\ii\ph} - ww' - \ii w^2\bigg)
  \nonumber\\[0.05cm]
= {} & \frac{T}{w^2}\,
  \bigg(w^2\,\frac{\psi^{(3)} - \ii\psi'' - 2\ii\psi'\psi''}{\psi'' - \ii\psi'(1 + \psi')}
  - ww' - \ii w^2\bigg)\nonumber\\[0.05cm]
= {} & \frac{T}{w^2}\,
  \bigg(w^2\,\frac{\psi^{(3)} - \ii\psi'' - 2\ii\psi'\psi''}{\psi'' - \ii\psi'(1 + \psi')}\,
  \frac{\psi'' + \ii\psi'(1 + \psi')}{\psi'' + \ii\psi'(1 + \psi')} - ww' - \ii w^2\bigg)
  \nonumber\\[0.05cm]
= {} & \frac{T}{w^2}
  \left[\left(\psi^{(3)} - \ii\psi'' - 2\ii\psi'\psi''\right)
  \left(\psi'' + \ii\psi' + \ii\psi'^2\right) - ww' - \ii w^2\right].
\end{align}
With
\beq
  w'
= \frac{\psi''\left(\psi'+3\psi'^2+2\psi'^3+\psi^{(3)}\right)}{w}\,,  
\eeq
\eqref{Eq:T'-1} becomes
\begin{align*}
  T'
= {} & T\,w^{-2}
  \left[\left(\psi'' + \ii\psi' + \ii\psi'^2\right)
  \left(\psi^{(3)} - \ii\psi'' - 2\ii\psi'\psi''\right)
  - \psi''\left(\psi' + 3\psi'^2 + 2\psi'^3 + \psi^{(3)}\right)\right.\\ 
& \qquad\quad\; - \ii\left(\psi''^2 + \psi'^2 + 2\psi'^3 + \psi'^4\right)\Big]\displaybreak[0]\\
= {} & T\,w^{-2}\,
  \Big(\psi'' \psi^{(3)} - \ii\psi''^2 - 2\ii\psi'\psi''^2
  + \ii\psi'\psi^{(3)} + \psi'\psi'' + 2\psi'^2\psi''
  + \ii\psi'^2\psi^{(3)} + \psi'^2\psi'' + 2\psi'^3\psi''\\ 
& \qquad\quad\; - \psi''\psi' - 3\psi''\psi'^2 - 2\psi''\psi'^3 - \psi''\psi^{(3)}
  - \ii\psi''^2 - \ii\psi'^2 - 2\ii\psi'^3 - \ii\psi'^4\Big)\displaybreak[0]\\
= {} & \ii\,T\,w^{-2}\,
  \Big(\psi'\psi^{(3)} + \psi'^2\psi^{(3)} - \psi'^2 - 2\psi'^3 - \psi'^4
  - 2\psi''^2 - 2\psi'\psi''^2\Big)\displaybreak[0]\\
= {} & \ii\,T\,w^{-2}
  \left[(1 + \psi')\psi'\psi^{(3)} - (1 + \psi')^2\psi'^2 
  - 2(1 + \psi')\psi''^2\right]\\
= {} & \ii\,T\,w^{-2} (1+\psi')
  \left[\psi'\left(\psi^{(3)} - \psi' - \psi'^2\right) - 2\psi''^2\right].
\end{align*}
So, with the function
\beqn \label{Eq:h}
\begin{aligned}
  h\, \colon [0,2\pi]
\;\rightarrow\; {} & \R\\ 
  \ph
\;\mapsto\; {} & h(\ph)
:= \frac{(1+\psi'(\ph))\left[\psi'(\ph)\left(\psi^{(3)}(\ph) - \psi'(\ph) - \psi'^2(\ph)\right) 
  - 2\psi''^2(\ph)\right]}
  {w^2(\ph)}\,,  
\end{aligned}
\eeqn
we have found
\beqn \label{Eq:T'-2}
  T'(\ph)
= \ii\,h(\ph)\,T(\ph)\,.  
\eeqn
Applying \eqref{Eq:T'-2}, we get
\beq
  \left[T,T'\right]
= [T,\, \ii h T]  
= h\,[T,\ii T]  
= h\,T\,\overline{T}\,[1,\ii]
= h
\eeq
and
\beq
  [\ii\,\ee^{\pm\ii\alpha}\,T',\, X_P]
= [\ii\,\ee^{\pm\ii\alpha}\,\ii h T,\, X_P]  
= [-\ee^{\pm\ii\alpha}\,h T,\, X_P]
= -h\,[\ee^{\pm\ii\alpha}\,T,\, X_P]\,,  
\eeq
hence \eqref{Eq:X-2} becomes
\begin{align*}
  X
= {} & \frac{([\ii\,\ee^{\pm\ii\alpha}\,T,\, X_P] - \la\cos\alpha)\,\ii h T
  + h\,[\ee^{\pm\ii\alpha}\,T,\, X_P]\,T}{h}\,
  \ii\,\ee^{\pm\ii\alpha}\displaybreak[0]\\[0.05cm]
= {} & \left\{[\ii\,\ee^{\pm\ii\alpha}\,T,\, X_P]\,\ii\,T - \ii\,\la\,T\cos\alpha
  + [\ee^{\pm\ii\alpha}\,T,\, X_P]\,T\right\}\ii\,\ee^{\pm\ii\alpha}\displaybreak[0]\\[0.05cm]
= {} & \left\{\!-\frac{1}{2}\left(\ii\,\ee^{\pm\ii\alpha}\,T\,\overline{X}_P 
  + \ii\,\ee^{\mp\ii\alpha}\,\overline{T}\,X_P\right)T - \ii\,\la\,T\cos\alpha
  + \frac{\ii}{2}\left(\ee^{\pm\ii\alpha}\,T\,\overline{X}_P
  - \ee^{\mp\ii\alpha}\,\overline{T}\,X_P\right)T\right\}\ii\,\ee^{\pm\ii\alpha}\\[0.05cm]    
= {} & \left\{\!-\frac{\ii}{2}\left(\ee^{\pm\ii\alpha}\,T^2\,\overline{X}_P 
  + \ee^{\mp\ii\alpha}\,T\,\overline{T}\,X_P\right) - \ii\,\la\,T\cos\alpha
  + \frac{\ii}{2}\left(\ee^{\pm\ii\alpha}\,T^2\,\overline{X}_P
  - \ee^{\mp\ii\alpha}\,T\,\overline{T}\,X_P\right)\right\}\ii\,\ee^{\pm\ii\alpha}\\[0.05cm]
= {} & \left\{-\ii\,\ee^{\mp\ii\alpha}\,X_P - \ii\,\la\,T\cos\alpha\right\}\ii\,\ee^{\pm\ii\alpha}\\[0.05cm]
= {} & X_P + \la\,T\,\ee^{\pm\ii\alpha}\cos\alpha\,,  
\end{align*}
We put $\XFkpm(\ph) := X$, and the proof of Theorem \ref{Thm:X_F} is complete.
\end{proof}

From Theorem \ref{Thm:X_F} with \eqref{Eq:tx_and_ty} one easily finds the result of Corollary \ref{Coro:xF_and_yF}.

\begin{coro} \label{Coro:xF_and_yF}
A parametric representation of the flank curves $\XFkpm$ of tooth $k$, $k = 1,2,\ldots,z_1$, of the drive gear is given by
\begin{align*}
  \xFkpm(\ph)
= {} & \phantom{-}r(\ph)\cos\ph - \lak(\ph)\cos\alpha\,
  \frac{\psi'(\ph)(1+\psi'(\ph))\sin(\ph\mp\alpha) - \psi''(\ph)\cos(\ph\mp\alpha)}
  	   {w(\ph)}\,,\\[0.1cm]
  \yFkpm(\ph)
= {} & {-r(\ph)\sin\ph} - \lak(\ph)\cos\alpha\,
  \frac{\psi'(\ph)(1+\psi'(\ph))\cos(\ph\mp\alpha) + \psi''(\ph)\sin(\ph\mp\alpha)}
  	   {w(\ph)}\,.
\end{align*}
\end{coro}

\begin{remark}
The first Frenet formula for plane curves (see e.\,g.\ \textcite[p.\ 10]{Kuehnel}) is
\beq
  \frac{\dd T}{\dd s}
= \kappa \ii T\,,
\eeq
where $\ii T$ is the unit normal vector which one gets by a counter-clockwise rotation of $T$ around $\pi/2$, $s$ the arc length, and $\kappa$ the {\em oriented} curvature. 
In the case of the curve $X_P$ we have 
\beqn \label{Eq:dT/ds}
  \frac{\dd T}{\dd s}
= \frac{\dd T}{\dd\ph}\,\frac{\dd\ph}{\dd s}
= \kappa \ii T
  \quad\Longrightarrow\quad
  T'(\ph)
= s'(\ph)\,\kappa(\ph)\,\ii\,T(\ph)\,.
\eeqn
Comparing \eqref{Eq:dT/ds} to \eqref{Eq:T'-2}, one sees that
\beqn \label{Eq:h=s'*kappa}
  h(\ph) = s'(\ph)\,\kappa(\ph)\,.
\eeqn  
With
\beqn \label{Eq:s'}
  s'(\ph)
= \frac{\dd}{\dd\ph} \left(a\int_{\ph_0}^\ph \frac{w(\phi)}{(1+\psi'(\phi))^2}\,\dd\phi\right)  
= \frac{a w(\ph)}{(1+\psi'(\ph))^2}
= |X_P'(\ph)|
\eeqn
it follows that the curvature of $X_P$ at point $\ph$ is
\beqn \label{Eq:kappa(phi)}
\begin{aligned}
  \kappa(\ph)
= {} & \frac{h(\ph)}{s'(\ph)}
= \frac{\left(1+\psi'(\ph)\right)^2 h(\ph)}{a w(\ph)}\\[0.1cm]
= {} & \frac{\left(1+\psi'(\ph)\right)^3 \left[\psi'(\ph)\left(\psi^{(3)}(\ph) - \psi'(\ph) - \psi'^2(\ph)\right) 
  - 2\psi''^2(\ph)\right]}
  {a w^3(\ph)}\,.  
\end{aligned}
\eeqn
Since it can be assumed that $X_P$ is a regular curve, we have $s'(\ph) = |X_P'(\ph)| > 0$.
If $\kappa(\ph) > 0$, then $T'(\ph)$ and $\ii T(\ph)$ have equal direction, hence $X_P$ turns to the left.
If $\kappa(\ph) < 0$, then $T'(\ph)$ and $\ii T(\ph)$ have opposite direction, hence $X_P$ turns to the right. \hfill\tr
\end{remark}

The perpendicular to the flank of the rack-cutter at the currently generated point $\XFkpm(\ph)$ of the gear tooth flank $\XFkpm$ passes through the instantaneous centre of velocity, $X_P(\ph)$, of the rack-cutter motion (see \cite[pp.\ 26-27]{Wunderlich}, \cite[pp.\ 274-275, 341-342] {Litvin&Fuentes}).
Knowing this, it is possible to give a rather short alternative proof of Theorem \ref{Thm:X_F} without using the envelope of the family of rack-cutter flank positions.

\begin{SCfigure}[][h]
  \includegraphics[scale=1]{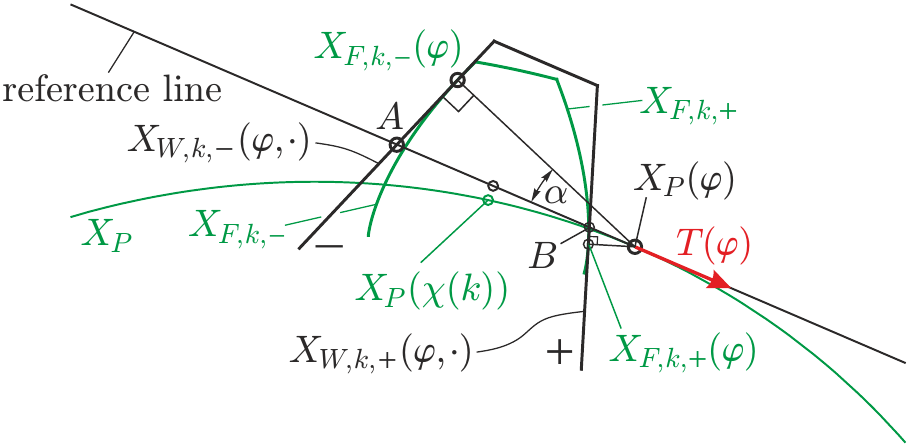}
  \caption{Sketch for the short proof the Theorem \ref{Thm:X_F}}
  \label{Fig:Triangle01}   
\end{SCfigure}

\begin{proof}[Short proof of Theorem \tn{\ref{Thm:X_F}}] 
We consider the situation in which point $\XFkm(\ph)$ of the flank curve $\XFkm$ is generated (see Fig.\ \ref{Fig:Triangle01}).
We denote by $A$ the intersection between the reference line and the line $\XWkm(\ph,\cdot)$, and by $\mu$ the signed distance between $A$ and $\XFkm(\ph) \equiv \XWkm(\ph,\mu)$ with $\mu > 0$ ($\mu < 0$) if $\XFkm(\ph)$ is on the left side (right side) of the reference line with respect to the direction of vector $T(\ph)$.
The signed distance between $X_P(\ph)$ and $A$ is given by $\la_{k,-}(\ph)$ (see \eqref{Eq:lambda_k} and Fig.\ \ref{Fig:Tooth_generation01a}).
The points $X_P(\ph)$, $A$ and $\XFkm(\ph)$ are the vertices of a rectangular triangle and thus easily follows 
\beqn \label{Eq:mu_minus}
  \mu
= -\la_{k,-}(\ph)\sin\alpha\,.  
\eeqn
Analogously one finds the different value 
\beqn \label{Eq:mu_plus}
  \mu
= \la_{k,+}(\ph)\sin\alpha  
\eeqn
for the point $\XFkp(\ph)$ of the flank curve $\XFkp$ by considering the rectangular triangle with vertices $X_P(\ph)$, $B$ and $\XFkp(\ph)$.
So we can write \eqref{Eq:mu_minus} and \eqref{Eq:mu_plus} together as  
\beqn \label{Eq:mu_plus_minus}
  \mu
= \pm\la_{k,\pm}(\ph)\sin\alpha\,.  
\eeqn
Inserting \eqref{Eq:mu_plus_minus} into \eqref{Eq:XWkpm}, we get
\begin{align*}
  \XFkpm(\ph)
= {} & \XWkpm(\ph,\pm\lak(\ph)\sin\alpha)
  \displaybreak[0]\\[0.05cm] 
= {} & X_P(\ph) + \lak(\ph)\left(1 \pm \ii\ee^{\pm\ii\alpha}\sin\alpha\right) T(\ph)
  \displaybreak[0]\\[0.05cm]
= {} & X_P(\ph) + \lak(\ph)\left[1 \pm \ii(\cos\alpha \pm \ii\sin\alpha)\sin\alpha\right] T(\ph)
  \displaybreak[0]\\[0.05cm]
= {} & X_P(\ph) + \lak(\ph)\left[1 \pm (\ii\cos\alpha\sin\alpha \mp \sin^2\alpha)\right] T(\ph)
  \displaybreak[0]\\[0.05cm]
= {} & X_P(\ph) + \lak(\ph)\left(1 - \sin^2\alpha \pm \ii\cos\alpha\sin\alpha\right) T(\ph)
  \displaybreak[0]\\[0.05cm]  
= {} & X_P(\ph) + \lak(\ph)\left(\cos^2\alpha \pm \ii\cos\alpha\sin\alpha\right) T(\ph)
  \displaybreak[0]\\[0.05cm]
= {} & X_P(\ph) + \lak(\ph)\cos\alpha\left(\cos\alpha \pm \ii\sin\alpha\right) T(\ph)
  \displaybreak[0]\\[0.05cm]
= {} & X_P(\ph) + \lak(\ph)\,T(\ph)\,\ee^{\pm\ii\alpha}\cos\alpha\,. \qedhere
\end{align*}
\end{proof}      

$X_P$ is a negatively oriented curve.
This means that an outer parallel curve is on its left side, and an inner parallel curve on its right side.
Hence, with $X_P(\ph)$ and $T(\ph)$ according to \eqref{Eq:X_P} and \eqref{Eq:T}, respectively, a parallel curve with distance $d$, $0 < d < \infty$, is given by
\begin{align} \label{Eq:Xpdpm}
  \Xpdpm(\ph)
= {} & X_P(\ph) + d\,T(\ph)\,\ee^{\pm\ii\pi/2}
= X_P(\ph) \pm d\,\ii\,T(\ph)\nonumber\\[0.1cm]
= {} & X_P(\ph) \pm d\,\ii\,
  \frac{\psi''(\ph) - \ii \psi'(\ph)(1+\psi'(\ph))}{w(\ph)}\,\ee^{-\ii\ph}\displaybreak[0]\nonumber\\[0.1cm]
= {} & X_P(\ph) \pm d\,\frac{\ii \psi''(\ph) + \psi'(\ph)(1+\psi'(\ph))}{w(\ph)}\,\ee^{-\ii\ph}\,,
\end{align}
where the upper and lower sign are for an outer and inner parallel curve, respectively; $w(\ph)$ see \eqref{Eq:w}.
Note that our convention concerning ``$\pm$'' does not apply to $\Xpdpm(\ph)$.
Splitting of \eqref{Eq:Xpdpm} into real and imaginary part yields 
\begin{align*}
  x_{p,d,\pm}(\ph)
= {} & \phantom{-}r(\ph)\cos\ph \pm d\,\frac{\psi''(\ph)\sin\ph + \psi'(\ph)(1+\psi'(\ph))\cos\ph}{w(\ph)}\,,\\[0.1cm]
  y_{p,d,\pm}(\ph)
= {} & {-r(\ph)\sin\ph} \pm d\,\frac{\psi''(\ph)\cos\ph - \psi'(\ph)(1+\psi'(\ph))\sin\ph}{w(\ph)}\,.   
\end{align*}

We denote by $\XMkpm(\ph)$ the parametric equation of the curve $\XMkpm$ (see Fig.\ \ref{Fig:Curves_for_tooth_k}) of the mid point $M_\pm$ (see Fig.\ \ref{Fig:Reference_profile01}).
Considering Fig.\ \ref{Fig:Reference_profile01}, we see that
\beqn \label{Eq:tan_alpha_and_cos_alpha}
  \tan\alpha = \frac{\ell_1}{\hf-\rh}
  \qquad\mbox{and}\qquad
  \cos\alpha = \frac{\rh}{\ell_2}\,. 
\eeqn
Now, using Fig.\ \ref{Fig:Tooth_generation01a} we find
\begin{align} \label{Eq:XMkpm}
  \XMkpm(\ph)
= {} & X_P(\ph) + \left[\pm\tfrac{\pi m}{4} \pm \ell_1 \pm \ell_2 
  - aI(\chi(k),\ph)\right] T(\ph) - \ii\,(\hf-\rh)\,T(\ph)\nonumber\\[0.05cm]  
= {} & X_P(\ph) + \left[\lak(\ph) \pm (\hf-\rh)\tan\alpha \pm \rh\sec\alpha\right] T(\ph)
  - \ii\,(\hf-\rh)\,T(\ph)\nonumber\\[0.05cm]
= {} & X_P(\ph) 
  + \left[\lak(\ph) \pm \rh\sec\alpha - (\hf-\rh)(\ii \mp \tan\alpha\right] T(\ph)\,.
\end{align}
The tangent unit vector at point $\XMkpm(\ph)$ of $\XMkpm$ is given by
\beq
  \TMkpm(\ph)
:= \frac{\XMkpm'(\ph)}{\,\big|\XMkpm'(\ph)\big|\,}\,.   
\eeq
From \eqref{Eq:XMkpm} with \eqref{Eq:T}, \eqref{Eq:T'-2} and \eqref{Eq:lambda_k'=-|X_P'|} we obtain
\begin{align*}
  \XMkpm'(\ph)
= {} & X_P'(\ph) + \lak'(\ph)\,T(\ph)
  + \left[\lak(\ph) \pm \rh\sec\alpha - (\hf-\rh)(\ii \mp \tan\alpha)\right] T'(\ph)\\[0.05cm]
= {} & \left\{|X_P'(\ph)| + \lak'(\ph)
  + \left[\lak(\ph) \pm \rh\sec\alpha - (\hf-\rh)(\ii \mp \tan\alpha)\right] \ii\,h(\ph)\right\} T(\ph)\\[0.05cm]
= {} & \left[\lak(\ph) \pm \rh\sec\alpha - (\hf-\rh)(\ii \mp \tan\alpha)\right] \ii\,h(\ph)\,T(\ph)\,.     
\end{align*}
The normal unit vector at point $\ph$ of the curve $\XMkpm$ pointing to the left is given by
\begin{align} \label{Eq:NMkpm}
  \NMkpm(\ph)
= {} & \frac{\XMkpm'(\ph)\,\ee^{\ii\pi/2}}{\big|\XMkpm'(\ph)\,\ee^{\ii\pi/2}\big|}
= \frac{\XMkpm'(\ph)\,\ii}{\big|\XMkpm'(\ph)\,\ii\big|}\nonumber\displaybreak[0]\\[0.1cm]
= {} & {-}\frac{\left[\lak(\ph) \pm \rh\sec\alpha 
  - (\hf-\rh)(\ii \mp \tan\alpha)\right] h(\ph)\,T(\ph)}
  {\big|\!\left[\lak(\ph) \pm \rh\sec\alpha
  - (\hf-\rh)(\ii \mp \tan\alpha)\right] h(\ph)\,T(\ph)\big|}\nonumber\\[0.1cm]
= {} & {-}\frac{\left[\lak(\ph) \pm \rh\sec\alpha 
  - (\hf-\rh)(\ii \mp \tan\alpha)\right] h(\ph)\,T(\ph)}
  {\left|\lak(\ph) \pm \rh\sec\alpha 
  - (\hf-\rh)(\ii \mp \tan\alpha)\right|\, |h(\ph)|}\nonumber\\[0.1cm]
= {} & {-}\sgn(h(\ph))\,\frac{\left[\lak(\ph) \pm \rh\sec\alpha 
  - (\hf-\rh)(\ii \mp \tan\alpha)\right] T(\ph)}
  {\left|\lak(\ph) \pm \rh\sec\alpha 
  - (\hf-\rh)(\ii \mp \tan\alpha)\right|}\nonumber\\[0.1cm]
= {} & {-}\sgn(h(\ph))\,\frac{\left[\lak(\ph) \pm \rh\sec\alpha 
  - (\hf-\rh)(\ii \mp \tan\alpha)\right] T(\ph)}
  {\:\sqrt{[\lak(\ph) \pm \rh\sec\alpha \pm (\hf-\rh)\tan\alpha]^2 + (\hf-\rh)^2}\:}\,.
\end{align}
Thus, a parametric equation of the parallel curve of interest (fillet curve), $\Xrkpm$ (see Fig.\ \ref{Fig:Curves_for_tooth_k}), with distance $\rh$ to the curve $\XMkpm$ is given by
\beqn \label{Eq:Xrkpm}
  \Xrkpm(\ph)
= \XMkpm(\ph) + \rh\NMkpm(\ph)\,.  
\eeqn
From \eqref{Eq:Xrkpm} with \eqref{Eq:XMkpm} and \eqref{Eq:NMkpm} one easily gets
\begin{align*}
  \xrkpm(\ph)
= {} & \phantom{-}r(\ph)\cos\ph + \left(1 - \frac{\rh\,\sgn(h(\ph))}
  {\:\sqrt{[\lak(\ph) \pm \rh\sec\alpha \pm (\hf-\rh)\tan\alpha]^2 + (\hf-\rh)^2}\:}\right)\\[0.1cm]
& \qquad\qquad\qquad \cdot\left\{[\lak(\ph) \pm \rh\sec\alpha \pm (\hf-\rh)\tan\alpha]\,t_x(\ph)
  + (\hf-\rh)\,t_y(\ph)\right\},\\[0.1cm]  
  \yrkpm(\ph)
= {} & {-}r(\ph)\sin\ph + \left(1 - \frac{\rh\,\sgn(h(\ph))}
  {\:\sqrt{[\lak(\ph) \pm \rh\sec\alpha \pm (\hf-\rh)\tan\alpha]^2 + (\hf-\rh)^2}\:}\right)\\[0.1cm]
& \qquad\qquad\qquad \cdot\left\{[\lak(\ph) \pm \rh\sec\alpha \pm (\hf-\rh)\tan\alpha]\,t_y(\ph)
  - (\hf-\rh)\,t_x(\ph)\right\}.
\end{align*}


{\bf Contact point of fillet curve $\boldsymbol{\Xrkpm}$ and dedendum curve $\boldsymbol{X_f}$.}
We determine a condition for the value of $\ph$ at the contact point of fillet curve $\Xrkm$ and dedendum curve $X_f := X_{p,\hf,-}$ (see Fig.\ \ref{Fig:Curves_for_tooth_k}; $\Xpdpm$ see \eqref{Eq:Xpdpm}).
This contact point is generated when the normal of the curve $X_P$ at point $X_P(\ph)$ passes through point $M_-$ and thus also through $C_-$ (see Fig.\ \ref{Fig:Contact_point01}).
\begin{SCfigure}[][h]
  \includegraphics[scale=1]{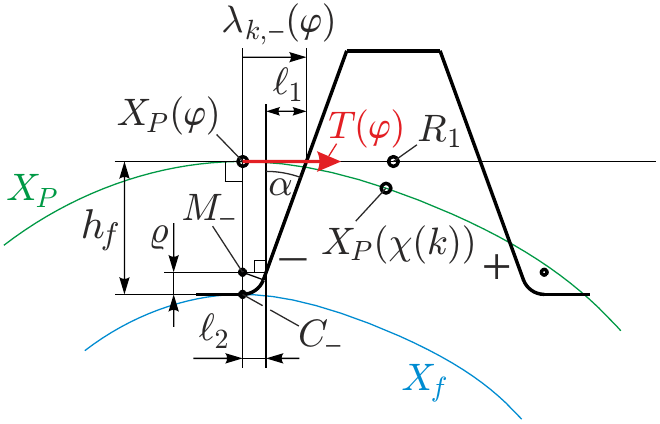}
  \caption{Sketch for the situation when the contact point of $\Xrkm$ (not shown) and $X_f$ is generated}
  \label{Fig:Contact_point01}
\end{SCfigure}
In this situation the point $C_-$ coincides with the contact point, and we have $\la_{k,-}(\ph) = \ell_1 + \ell_2$, hence (see Fig.\ \ref{Fig:Reference_profile01} and Eq.\ \eqref{Eq:tan_alpha_and_cos_alpha})
\beqn \label{Eq:-lambda_k}
  \la_{k,-}(\ph)
= (\hf-\rh)\tan\alpha + \rh\sec\alpha
\eeqn
follows as condition for the searched value of $\ph$.
Analogously, for the contact point of the fillet curve $\Xrkp$ and the dedendum curve $X_f$ the formula
\beqn \label{Eq:+lambda_k}
  \la_{k,+}(\ph)
= -(\hf-\rh)\tan\alpha - \rh\sec\alpha
\eeqn
is obtained.
We can combine \eqref{Eq:-lambda_k} and \eqref{Eq:+lambda_k} to
\beqn \label{Eq:-+lambda_k}
  {\mp}\lak(\ph)
= (\hf-\rh)\tan\alpha + \rh\sec\alpha\,.
\eeqn

\section{Driven gear}
\label{Sec:Driven_gear}

Since both flank curves of a tooth of the drive gear are meshing with both flank curves of a tooth space of the driven gear, here we consider the flanks of the tooth spaces (see Fig.\ \ref{Fig:Definition_flanks01}).
For the generation of a tooth of the drive gear we used the flanks of a tooth space of the rack-cutter.
Now, we use these flanks of the rack-cutter again, but as flanks of a tooth.
This approach has the advantage that we can determine corresponding points of the flanks of the drive and the driven gear; that means: points that are in contact for a value of the drive angle $\ph$.        

From \eqref{Eq:r_and_R} and \eqref{Eq:Xi_P} we get
\beq
  R'(\ph)
= - \frac{a\psi''(\ph)}{\ds (1+\psi'(\ph))^2}
  \quad\mbox{and}\quad
  \varXi_P'(\ph)
= -R'(\ph)\ee^{\ii \psi(\ph)} - R(\ph)\ii \psi'(\ph)\ee^{\ii \psi(\ph)}\,,    
\eeq
respectively, hence
\beq
  \varXi_P'(\ph)
= \frac{a\psi''(\ph)}{\ds (1+\psi'(\ph))^2}\,\ee^{\ii \psi(\ph)}
  - \frac{a\ii \psi'(\ph)}{\ds 1+\psi'(\ph)}\,\ee^{\ii \psi(\ph)}
= \frac{a\psi''(\ph)-a\ii \psi'(\ph)(1+\psi'(\ph))}{\ds (1+\psi'(\ph))^2}\,\ee^{\ii \psi(\ph)}   
\eeq
and
\begin{align} \label{Eq:|Xi_P'|}
  |\varXi_P'(\ph)|
= {} & \left|\frac{a\psi''(\ph)}{\ds (1+\psi'(\ph))^2} - \frac{a\ii \psi'(\ph)}{\ds 1+\psi'(\ph)}\right|
  \Big|\ee^{\ii \psi(\ph)}\Big|  
= \sqrt{\frac{a^2\psi''^2(\ph)}{(1+\psi'(\ph))^4} + \frac{a^2\psi'^2(\ph)}{(1+\psi'(\ph))^2}}
  \nonumber\\[0.1cm] 
= {} & \frac{a\,\sqrt{\psi''^2(\ph) + \psi'^2(\ph)(1+\psi'(\ph))^2}}{(1+\psi'(\ph))^2}
= \frac{aw(\ph)}{(1+\psi'(\ph))^2}\,,
\end{align}
$w(\ph)$ see \eqref{Eq:w}.
It follows that the tangent unit vector of $\varXi_P$ at point $\varXi_P(\ph)$ (see Fig.\ \ref{Fig:Tooth_generation01b}) is given by
\beqn \label{Eq:TXi}
  \TXi(\ph)
:= \frac{\varXi_P'(\ph)}{|\varXi_P'(\ph)|}
= \frac{\psi''(\ph) - \ii \psi'(\ph)(1+\psi'(\ph))}{w(\ph)}\,\ee^{\ii \psi(\ph)}\,.
\eeqn
Splitting \eqref{Eq:TXi} into real and imaginary part gives 
\beqn \label{Eq:txi_and_teta}
\left.
\begin{aligned}
  t_\xi(\ph) := \Rez \TXi(\ph)   
= {} & \frac{\psi''(\ph)\cos\psi(\ph) + \psi'(\ph)(1+\psi'(\ph))\sin\psi(\ph)}{w(\ph)}\,,\\[0.1cm]
  t_\eta(\ph) := \Imz \TXi(\ph)
= {} & \frac{\psi''(\ph)\sin\psi(\ph) - \psi'(\ph)(1+\psi'(\ph))\cos\psi(\ph)}{w(\ph)}\,.      
\end{aligned}
\;\right\}
\eeqn

Now we determine a parametric equation for the rack-cutter flank line $\XiWkpm$ that generates flank curve $\XiFkpm$ of tooth space $k$, $k = 1,2,\ldots,z_2$. 
If $\ph = \chi(k)$, then the point $R_2$ of the rack-cutter coincides with the point $\varXi_P(\chi(k))$ of the centrode $\varXi_P$ (see Fig.\ \ref{Fig:Tooth_generation01b}).
The vectors $\TXi(\ph)$ and $\vv{\varXi_P(\ph)R_2}$ point into the same direction if $I(\chi(k),\ph) < 0$.
The signed arc length between the points $\varXi_P(\chi(k))$ and $\varXi_P(\ph)$ ($=$ signed distance between $R_2$ and $\varXi_P(\ph)$) is equal to $aI(\chi(k),\ph)$.
The signed distances $\la_{k,+}(\ph)$ and $\la_{k,-}(\ph)$ are given by \eqref{Eq:lambda_k}.
$\lak(\ph) > 0$ indicates that $\TXi(\ph)$ and the arrow of $\lak(\ph)$ have the same direction.   
So we have the following parametric equation for the rack-cutter flank line $\XiWkpm$
\beqn
\begin{aligned} \label{Eq:XiWkpm}
  \XiWkpm(\ph,\mu)
= {} & \varXi_P(\ph) + \lak(\ph)\,\TXi(\ph)
  + \mu\,\TXi(\ph)\,\ee^{\ii\left(\frac{\pi}{2}\,\pm\,\alpha\right)}\\[0.05cm]
= {} & \varXi_P(\ph) + \lak(\ph)\,\TXi(\ph) + \mu\,\TXi(\ph)\,\ii\ee^{\pm\ii\alpha}\\[0.05cm] 
= {} & \varXi_P(\ph) + \left[\lak(\ph) + \mu\ii\ee^{\pm\ii\alpha}\right] \TXi(\ph),
  \quad \mu \in \R\,.    
\end{aligned}
\eeqn
Splitting of \eqref{Eq:XiWkpm} into real and imaginary part gives
\beqn \label{Eq:xiWk_and_etaWk}
\left.
\begin{aligned}
  \xiWkpm(\ph,\mu)
= {} & {-}R(\ph)\cos\psi(\ph) + \lak(\ph)\,t_\xi(\ph)
  + \mu\, [\mp t_\xi(\ph)\sin\alpha - t_\eta(\ph)\cos\alpha]\,,\\[0.1cm]
  \etaWkpm(\ph,\mu)
= {} & \hspace{0.025cm}{-}R(\ph)\sin\psi(\ph) + \lak(\ph)\,t_\eta(\ph)
  + \mu\, [\mp t_\eta(\ph)\sin\alpha + t_\xi(\ph)\cos\alpha]\,.   
\end{aligned}
\;\right\}
\eeqn

\begin{SCfigure}[][h]
  \includegraphics[scale=1]{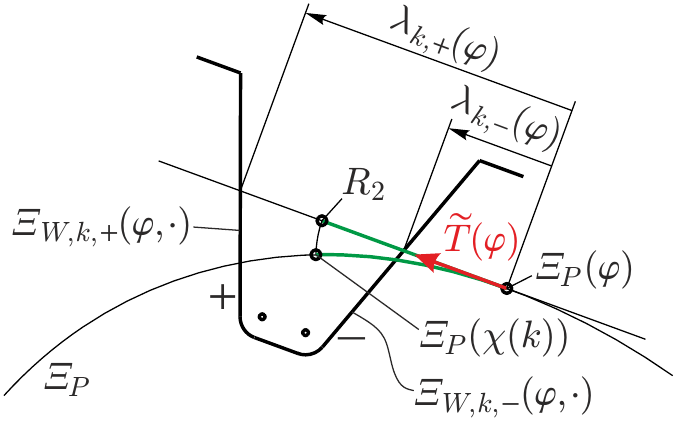}
  \caption{Motion of the rack-cutter during the generation of tooth space $k$}
  \label{Fig:Tooth_generation01b} 
\end{SCfigure}

\begin{thm} \label{Thm:Xi_F}
A parametric equation of the flank curves $\XiFkpm$ of tooth space $k$, $k = 1,2,\ldots,z_2$, of the driven gear is given by
\beqn \label{Eq:XiFkpm}
  \XiFkpm(\ph)
= \varXi_P(\ph) + \lak(\ph)\,\TXi(\ph)\,\ee^{\pm\ii\alpha}\cos\alpha\,,
\eeqn
where
\begin{gather*}
  \varXi_P(\ph)
= -R(\ph)\,\ee^{\ii \psi(\ph)}\,,\qquad
  \TXi(\ph)
= \frac{\psi''(\ph) - \ii \psi'(\ph)(1+\psi'(\ph))}{w(\ph)}\,\ee^{\ii \psi(\ph)}\,,\\[0.1cm]
  \lak(\ph)
= \pm\frac{\pi m}{4} - a \int_{\chi(k)}^\ph \frac{w(\phi)}{(1+\psi'(\phi))^2}\,\dd\phi  
\end{gather*}
with
\beq
  R(\ph)
= \frac{a}{1 + \psi'(\ph)}\,,\qquad
  w(\ph)
= \sqrt{\strut\ds \psi''^2(\ph) + \psi'^2(\ph)(1+\psi'(\ph))^2}\,,
\eeq
and $\chi(k)$ according to \eqref{Eq:chi(k)}.
\end{thm}

\begin{proof}
The proof is completely analogous to that of Theorem \ref{Thm:X_F}.
We start with the linear equation system
\beq
\left.
\begin{array}{r@{\;=\;}l}
  [A(\ph)-B(\ph),\varXi] & [A(\ph),B(\ph)]\:,\\[0.1cm]
  [(A(\ph)-B(\ph))',\varXi] & [A(\ph),B(\ph)]',
\end{array}  
\right\}
\eeq
choose
\begin{align*}
  A = A(\ph)
:= {} & \XiWkpm(\ph,1)
= \varXi_P(\ph) + \lak(\ph)\,\TXi(\ph) + \ii\,\ee^{\pm\ii\alpha}\,\TXi(\ph)\,,\\[0.05cm] 
  B = B(\ph)
:= {} & \XiWkpm(\ph,0)
= \varXi_P(\ph) + \lak(\ph)\,\TXi(\ph)\,,
\end{align*}
and get
\begin{gather*}
  A - B
= \ii\,\ee^{\pm\ii\alpha}\,\TXi(\ph)\,,\quad
  (A - B)' = \ii\,\ee^{\pm\ii\alpha}\,\TXi'(\ph)\,,\quad
  [A-B,(A-B)'] = \big[\TXi(\ph),\TXi'(\ph)\big]\,,\\[0.1cm]
  [A,B]
= \big[\ii\,\ee^{\pm\ii\alpha}\,\TXi(\ph),\, \varXi_P(\ph)\big] - \lak(\ph)\cos\alpha\,,\\[0.1cm]
  [A,B]'
= \big[\ii\,\ee^{\pm\ii\alpha}\,\TXi'(\ph),\, \varXi_P(\ph)\big]
  - |\varXi_P'|\cos\alpha - \lak'(\ph)\cos\alpha\,.   
\end{gather*}
With
\beqn \label{Eq:lambda_k'=-|Xi_P'|}
  \lak'(\ph) = -|\varXi_P'(\ph)|
\eeqn
(see \eqref{Eq:lambda_k'} and \eqref{Eq:|Xi_P'|}) we have
\beq
  [A,B]'
= \big[\ii\,\ee^{\pm\ii\alpha}\,\TXi'(\ph),\, \varXi_P(\ph)\big]\,.
\eeq
It follows that (cp.\ \eqref{Eq:X-2})
\beqn \label{Eq:Xi}
  \varXi
= \frac
	{\big(\big[\ii\,\ee^{\pm\ii\alpha}\,\TXi(\ph),\, \varXi_P(\ph)\big]
		- \lak(\ph)\cos\alpha\big)\,\TXi'(\ph) 
		- [\ii\,\ee^{\pm\ii\alpha}\,\TXi'(\ph),\, \varXi_P(\ph)]\,\TXi(\ph)}
	{\rule[1.5ex]{0pt}{1ex}\tn{$\big[\TXi(\ph),\, \TXi'(\ph)\big]$}}
  \,\ii\,\ee^{\pm\ii\alpha}\,. 
\eeqn
From \eqref{Eq:TXi} one finds
\beqn \label{Eq:TXi'}
  \TXi'(\ph)
= \ii\,\tilde{h}(\ph)\,\TXi(\ph)  
\eeqn
with the function
\beqn \label{Eq:hXi}
\begin{aligned}
  \hXi\, \colon [0,2\pi] 
\;\rightarrow\; {} & \R\\ 
  \ph
\;\mapsto\; {} & \hXi(\ph)
:= \frac{(1+\psi'(\ph))\left[\psi'(\ph)\left(\psi^{(3)}(\ph) + \psi'^2(\ph) + \psi'^3(\ph)\right) 
  - \psi''^2(\ph)\right]}
  {w^2(\ph)}\,.  
\end{aligned}
\eeqn
Plugging \eqref{Eq:TXi'} into \eqref{Eq:Xi} finally gives
\beq
  \varXi
= \varXi_P(\ph) + \lak(\ph)\,\TXi(\ph)\,\ee^{\pm\ii\alpha}\cos\alpha\,.
\eeq
We put $\XiFkpm(\ph) := \varXi$ which completes the proof of Theorem \ref{Thm:Xi_F}.
\end{proof}

From Theorem \ref{Thm:Xi_F} with \eqref{Eq:txi_and_teta} one concludes the result of Corollary \ref{Coro:xiF_and_etaF}.

\begin{coro} \label{Coro:xiF_and_etaF}
A parametric representation of the flank curves $\XiFkpm$ of tooth space $k$, $k = 1,2,\ldots,z_2$, of the driven gear is given by
\begin{align*}
  \xiFkpm(\ph)
= {} & {-}R(\ph)\cos\psi(\ph)\\[0.1cm]
& + \lak(\ph)\cos\alpha\,
  \frac{\psi''(\ph)\cos(\psi(\ph)\pm\alpha) + \psi'(\ph)(1+\psi'(\ph))\sin(\psi(\ph)\pm\alpha)}
  	   {w(\ph)}\,,\\[0.3cm]
  \etaFkpm(\ph)
= {} & {-R(\ph)\sin\psi(\ph)}\\[0.1cm]
& + \lak(\ph)\cos\alpha\,
  \frac{\psi''(\ph)\sin(\psi(\ph)\pm\alpha) - \psi'(\ph)(1+\psi'(\ph))\cos(\psi(\ph)\pm\alpha)}
  	   {w(\ph)}\,. 
\end{align*}
\end{coro}

\begin{remark}
The first Frenet formula in the case of the curve $\varXi_P$ is
\beq
  \frac{\dd \TXi}{\dd s}
= \kaXi\,\ii\,\TXi\,,
\eeq
where $s$ is the arc length and $\kaXi$ the {\em oriented} curvature.
We have
\beqn \label{Eq:dTXi/ds}
  \frac{\dd\TXi}{\dd s}
= \frac{\dd\TXi}{\dd\ph}\,\frac{\dd\ph}{\dd s}
= \kappa \ii \TXi
  \quad\Longrightarrow\quad
  \TXi'(\ph)
= s'(\ph)\,\kaXi(\ph)\,\ii\,\TXi(\ph)\,.
\eeqn
Comparing \eqref{Eq:dTXi/ds} to \eqref{Eq:TXi'}, one sees that
\beqn \label{Eq:hXi=s'*kaXi}
  \hXi(\ph) = s'(\ph)\,\kaXi(\ph)\,.
\eeqn
With
\beqn \label{Eq:s'(phi)=|Xi_P'(phi)|}
  s'(\ph) = |\varXi_P'(\ph)|
\eeqn
(see \eqref{Eq:s'} and \eqref{Eq:|Xi_P'|}), it follows that the curvature of $\varXi_P$ at point $\ph$ is
\beqn \label{Eq:kaXi(phi)}
\begin{aligned}
  \kaXi(\ph)
= {} & \frac{\hXi(\ph)}{s'(\ph)}
= \frac{\left(1+\psi'(\ph)\right)^2 \hXi(\ph)}{a w(\ph)}\\[0.1cm]
= {} & \frac{\left(1+\psi'(\ph)\right)^3 \left[\psi'(\ph)\left(\psi^{(3)}(\ph) + \psi'^2(\ph) + \psi'^3(\ph)\right) 
  - \psi''^2(\ph)\right]}
  {a w^3(\ph)}\,.  
\end{aligned}
\eeqn
Since it can be assumed that $\varXi_P$ is a regular curve, we have $s'(\ph) = |\varXi_P'(\ph)| > 0$.
If $\kaXi(\ph) > 0$, then $\TXi'(\ph)$ and $\ii \TXi(\ph)$ have equal direction, hence $\varXi_P$ turns to the left.
If $\kaXi(\ph) < 0$, then $\TXi'(\ph)$ and $\ii \TXi(\ph)$ have opposite direction, hence $\varXi_P$ turns to the right. \hfill\tr
\end{remark}

As for Theorem \ref{Thm:X_F}, we give a short proof of Theorem \ref{Thm:Xi_F} which uses the fact that the perpendicular to the flank of the rack-cutter at the currently generated point of the gear tooth flank passes through the instantaneous centre of velocity of the motion of the rack-cutter.

\begin{proof}[Short proof of Theorem {\em \ref{Thm:Xi_F}}]
The value of the parameter $\mu$ at the contact point $\XiWkpm(\ph,\mu) \equiv \XiFkpm(\ph)$ between rack-cutter flank line $\XiWkpm(\ph,\cdot)$ and tooth flank curve $\XiFkpm$ to be generated is given by \eqref{Eq:mu_plus_minus}.
Therefore, applying \eqref{Eq:XiWkpm}, we get
\begin{align*}
  \XiFkpm(\ph)
= {} & \XiWkpm(\ph,\pm\lak(\ph)\sin\alpha)\displaybreak[0]\\[0.05cm] 
= {} & \varXi_P(\ph) + \lak(\ph)\left(1 \pm \ii\ee^{\pm\ii\alpha}\sin\alpha\right) \TXi(\ph)
  \displaybreak[0]\\[0.05cm]
= {} & \varXi_P(\ph) + \lak(\ph)\,\TXi(\ph)\,\ee^{\pm\ii\alpha}\cos\alpha\,. \qedhere  
\end{align*}
\end{proof}

$\varXi_P$ is a positively oriented curve.
An outer parallel curve is on its right side, and an inner parallel curve on its left side.
Therefore, with $\varXi_P(\ph)$ and $\TXi(\ph)$ according to \eqref{Eq:Xi_P} and \eqref{Eq:TXi}, respectively, a parallel curve with distance $d$, $0 < d < \infty$, is given by
\begin{align} \label{Eq:Xipdpm}
  \Xipdpm(\ph)
= {} & \varXi_P(\ph) + d\,\TXi(\ph)\,\ee^{\mp\ii\pi/2}
= \varXi_P(\ph) \mp d\,\ii\,\TXi(\ph)\nonumber\displaybreak[0]\\[0.1cm]
= {} & \varXi_P(\ph) \mp 
  d\,\ii\,\frac{\psi''(\ph) - \ii \psi'(\ph)(1+\psi'(\ph))}{w(\ph)}\,\ee^{\ii \psi(\ph)}
  \displaybreak[0]\nonumber\\[0.1cm]
= {} & \varXi_P(\ph) \mp 
  d\,\frac{\psi'(\ph)(1+\psi'(\ph)) + \ii \psi''(\ph)}{w(\ph)}\,\ee^{\ii \psi(\ph)}\,,  
\end{align}
where the upper and lower sign are for the outer and inner parallel curve, respectively; $w(\ph)$ see \eqref{Eq:w}.
Note that our convention concerning ``$\pm$'' does not apply to $\Xipdpm(\ph)$. 
Splitting of \eqref{Eq:Xipdpm} into real and imaginary part yields 
\begin{align*}
  \xi_{p,d,\pm}(\ph)
= {} & {-R(\ph)\cos \psi(\ph)} \mp 
  d\,\frac{\psi'(\ph)(1+\psi'(\ph))\cos \psi(\ph) - \psi''(\ph)\sin \psi(\ph)}{w(\ph)}\,,\\[0.1cm]
  \eta_{p,d,\pm}(\ph)
= {} & {-R(\ph)\sin \psi(\ph)} \mp
  d\,\frac{\psi'(\ph)(1+\psi'(\ph))\sin \psi(\ph) + \psi''(\ph)\cos \psi(\ph)}{w(\ph)}\,.   
\end{align*}

We denote by $\XiMkpm(\ph)$ the parametric equation of the curve $\XiMkpm$ of the mid point $M_\pm$ in Fig.\ \ref{Fig:Reference_profile01} where $M_-$ is the right one.
We have
\begin{align} \label{Eq:XiMkpm}
  \XiMkpm(\ph)
= {} & \varXi_P(\ph) + \left[\pm\frac{\pi m}{4} \mp (\hf-\rh)\tan\alpha \mp \rh\sec\alpha
  - aI(\chi(k),\ph)\right] \TXi(\ph)\nonumber\\
& \hspace{1.085cm} + \ii\,(\hf-\rh)\,\TXi(\ph)\nonumber\\[0.05cm]
= {} & \varXi_P(\ph) 
  + \left[\lak(\ph) \mp \rh\sec\alpha + (\hf-\rh)(\ii \mp \tan\alpha\right] \TXi(\ph)\,.
\end{align}
The tangent unit vector at point $\XiMkpm(\ph)$ of $\XiMkpm$ is given by
\beq
  \TXiMkpm(\ph)
:= \frac{\XiMkpm'(\ph)}{\,\big|\XiMkpm'(\ph)\big|\,}\,.   
\eeq
From \eqref{Eq:XiMkpm} with \eqref{Eq:TXi}, \eqref{Eq:TXi'} and \eqref{Eq:lambda_k'=-|Xi_P'|} we obtain
\begin{align*}
  \XiMkpm'(\ph)
= {} & \varXi_P'(\ph) + \lak'(\ph)\,\TXi(\ph)
  + \left[\lak(\ph) \mp \rh\sec\alpha + (\hf-\rh)(\ii \mp \tan\alpha)\right] \TXi'(\ph)\\[0.05cm]
= {} & \big\{|\varXi_P'(\ph)| + \lak'(\ph)
  + \left[\lak(\ph) \mp \rh\sec\alpha + (\hf-\rh)(\ii \mp \tan\alpha)\right] \ii\,\hXi(\ph)\big\}\,
  \TXi(\ph)\\[0.05cm]
= {} & \left[\lak(\ph) \mp \rh\sec\alpha + (\hf-\rh)(\ii \mp \tan\alpha)\right] \ii\,\hXi(\ph)\,\TXi(\ph)\,.     
\end{align*}
The normal unit vector at point $\ph$ of the curve $\XiMkpm$ pointing to the right is given by
\begin{align} \label{Eq:NXi_M}
  \NXiMkpm(\ph)
= {} & \frac{\XiMkpm'(\ph)\,\ee^{-\ii\pi/2}}{\big|\XiMkpm'(\ph)\,\ee^{-\ii\pi/2}\big|}
= -\frac{\XiMkpm'(\ph)\,\ii}{\big|\XiMkpm'(\ph)\,\ii\big|}\nonumber\displaybreak[0]\\[0.1cm]
= {} & \frac{\left[\lak(\ph) \mp \rh\sec\alpha 
  + (\hf-\rh)(\ii \mp \tan\alpha)\right] \hXi(\ph)\,\TXi(\ph)}
  {\big|\!\left[\lak(\ph) \mp \rh\sec\alpha
  + (\hf-\rh)(\ii \mp \tan\alpha)\right] \hXi(\ph)\,\TXi(\ph)\big|}\displaybreak[0]\nonumber\\[0.1cm]
= {} & \frac{\left[\lak(\ph) \mp \rh\sec\alpha 
  + (\hf-\rh)(\ii \mp \tan\alpha)\right] \hXi(\ph)\,\TXi(\ph)}
  {\left|\lak(\ph) \mp \rh\sec\alpha 
  + (\hf-\rh)(\ii \mp \tan\alpha)\right|\, |\hXi(\ph)|}\displaybreak[0]\nonumber\\[0.1cm]
= {} & \sgn(\hXi(\ph))\,\frac{\left[\lak(\ph) \mp \rh\sec\alpha 
  + (\hf-\rh)(\ii \mp \tan\alpha)\right] \TXi(\ph)}
  {\left|\lak(\ph) \mp \rh\sec\alpha 
  + (\hf-\rh)(\ii \mp \tan\alpha)\right|}\nonumber\\[0.1cm]
= {} & \sgn(\hXi(\ph))\,\frac{\left[\lak(\ph) \mp \rh\sec\alpha 
  + (\hf-\rh)(\ii \mp \tan\alpha)\right] \TXi(\ph)}
  {\:\sqrt{[\lak(\ph) \mp \rh\sec\alpha \mp (\hf-\rh)\tan\alpha]^2 + (\hf-\rh)^2}\:}\,.
\end{align}
Thus,
\beqn \label{Eq:Xirkpm}
  \Xirkpm(\ph)
= \XiMkpm(\ph) + \rh\NXiMkpm(\ph)  
\eeqn
is a parametric equation of the parallel curve of interest (fillet curve), $\Xirkpm$, with distance $\rh$ to curve $\XiMkpm$.
From \eqref{Eq:Xirkpm} with \eqref{Eq:XiMkpm} and \eqref{Eq:NXi_M} one easily gets
\begin{align*}
  \xirkpm(\ph)
= {} & {-}R(\ph)\cos\psi(\ph) + \left(1 + \frac{\rh\,\sgn(\hXi(\ph))}
  {\:\sqrt{[\lak(\ph) \mp \rh\sec\alpha \mp (\hf-\rh)\tan\alpha]^2 + (\hf-\rh)^2}\:}\right)\\[0.1cm]
& \qquad\qquad\qquad\quad\;\;\, \cdot\left\{[\lak(\ph) \mp \rh\sec\alpha \mp (\hf-\rh)\tan\alpha]\,t_\xi(\ph)
  - (\hf-\rh)\,t_\eta(\ph)\right\},\\[0.1cm]  
  \etarkpm(\ph)
= {} & {-}R(\ph)\sin\psi(\ph) + \left(1 + \frac{\rh\,\sgn(\hXi(\ph))}
  {\:\sqrt{[\lak(\ph) \mp \rh\sec\alpha \mp (\hf-\rh)\tan\alpha]^2 + (\hf-\rh)^2}\:}\right)\\[0.1cm]
& \qquad\qquad\qquad\quad\;\;\, \cdot\left\{[\lak(\ph) \mp \rh\sec\alpha \mp (\hf-\rh)\tan\alpha]\,t_\eta(\ph)
  + (\hf-\rh)\,t_\xi(\ph)\right\}.
\end{align*}


{\bf Contact point of fillet curve $\boldsymbol{\Xirkpm}$ and dedendum curve $\boldsymbol{\varXi_f}$.}
One finds
\beqn \label{Eq:+-lambda_k}
  {\pm}\lak(\ph)
= (\hf-\rh)\tan\alpha + \rh\sec\alpha
\eeqn
as condition for the value of $\ph$ at the contact point of $\Xirkpm$ and $\varXi_f:=\varXi_{p,\hf,-}$ ($\Xipdpm$ see \eqref{Eq:Xipdpm}).


\section{Undercut}\label{Sec:Undercut}
\subsection{Drive gear (undercut)}

There are two cases for the transition between the fillet curve $\Xrkpm$ and the working part of the flank curve $\XFkpm$: \mynobreakpar
\begin{itemize}[leftmargin=0.6cm]
\setlength{\itemsep}{-2pt}
\item[a)] $\Xrkpm$ and the working part of $\XFkpm$ have a common tangent direction at the {\em transition point}.
It is said that the flank is free of {\em undercut}.
\item[b)] $\Xrkpm$ and the working part $\XFkpm$ have no common tangent direction at the {\em transition point}.
It is said that {\em undercut} occurs at the flank.
\end{itemize}

Normally, there exists always exactly one {\em contact point} of $\Xrkpm$ and $\XFkpm$ where both curves have a common tangent direction.
In case (a) the transition point is the contact point, in case (b) not.   

We consider the ``$-$''-side of the gear tooth.
Undercut occurs if the value $\phBkm$ of the parameter $\ph$ at the contact point of $\XFkm$ and $\Xrkm$ is smaller than then value $\phSkm$ of $\ph$ at the singular point (cusp) of $\XFkm$.
On the ``$+$''-side we have the reverse situation: Undercut occurs if the value $\phBkp$ of $\ph$ at the contact point of $\XFkp$ and $\Xrkp$ is greater than the value $\phSkp$ of $\ph$ at the singular point (cusp) of $\XFkp$.
One gets the values $\phSkpm$ by deriving $\XFkpm(\ph)$ (see \eqref{Eq:XFkpm}) with respect to $\ph$, and then setting the result equal to zero,
\beqn \label{Eq:XFkpm'(phi)=0}
  \XFkpm'(\phSkpm)
= \frac{\dd\XFkpm}{\dd\ph}\,(\ph_{S,\pm})
= 0\,.  
\eeqn
Applying \eqref{Eq:T}, \eqref{Eq:lambda_k'=-|X_P'|} and \eqref{Eq:T'-2}, we have
\begin{align*}
  \XFkpm'(\ph)
= {} & X_P'(\ph) + \left(\lak'(\ph)\,T(\ph) + \lak(\ph)\,T'(\ph)\right) \ee^{\pm\ii\alpha}\cos\alpha
  \displaybreak[0]\\[0.05cm]
= {} & |X_P'(\ph)|\,T(\ph) + \left(\lak'(\ph) + \lak(\ph)\,\ii h(\ph)\right)
  T(\ph)\,\ee^{\pm\ii\alpha}\cos\alpha\displaybreak[0]\\[0.05cm] 
= {} & \left[-\lak'(\ph) + \left(\lak'(\ph) + \lak(\ph)\,\ii h(\ph)\right)
  (\cos\alpha \pm \ii\sin\alpha)\cos\alpha\right] T(\ph)\\[0.05cm]
= {} & \left[-\lak'(\ph) + \lak'(\ph)\left(\cos^2\alpha \pm \ii\sin\alpha\cos\alpha\right)\right.\\
& \:\, + \left.\lak(\ph)\,h(\ph)\left(\ii\cos^2\alpha \mp \sin\alpha\cos\alpha\right)\right] T(\ph)
  \displaybreak[0]\\[0.05cm]
= {} & \left[-\lak'(\ph) + \lak'(\ph)\cos^2\alpha \mp \lak(\ph)\,h(\ph)\sin\alpha\cos\alpha\right.\\
& \:\, + \left.\ii\left(\lak(\ph)\,h(\ph)\cos^2\alpha \pm \lak'(\ph)\sin\alpha\cos\alpha\right)\right]
  T(\ph)\displaybreak[0]\\[0.05cm]
= {} & \left[-\lak'(\ph)\sin^2\alpha \mp \lak(\ph)\,h(\ph)\sin\alpha\cos\alpha\right.\\
& \:\, + \left.\ii\left(\lak(\ph)\,h(\ph)\cos^2\alpha \pm \lak'(\ph)\sin\alpha\cos\alpha\right)\right]
  T(\ph)\displaybreak[0]\\[0.05cm]
= {} & \left[-\!\left(\lak'(\ph)\sin\alpha \pm \lak(\ph)\,h(\ph)\cos\alpha\right)\sin\alpha\right.\\
& \:\, + \left.\ii\left(\lak(\ph)\,h(\ph)\cos\alpha \pm \lak'(\ph)\sin\alpha\right)\cos\alpha\right]
  T(\ph)\,.
\end{align*}
From the condition \eqref{Eq:XFkpm'(phi)=0} it follows that the term in the square brackets is equal to zero, and so real and imaginary part of this term are equal to zero which gives
\beq
  \lak(\ph)\,h(\ph)\cos\alpha \pm \lak'(\ph)\sin\alpha = 0\,, 
\eeq
hence
\beq
  \lak(\ph)\,h(\ph) = \mp \lak'(\ph)\tan\alpha\,.
\eeq
Using \eqref{Eq:h=s'*kappa} and $s'(\ph) = -\lak'(\ph)$ (cp.\ \eqref{Eq:lambda_k'} and \eqref{Eq:s'}), we get
\beq
  -\lak(\ph)\,\lak'(\ph)\,\kappa(\ph)
= \mp\lak'(\ph)\,\tan\alpha\,.
\eeq
So, we have
\beqn \label{Eq:lambda*kappa}
  \lak(\ph)\,\kappa(\ph) = \pm\tan\alpha
\eeqn
as condition for the value $\phSkpm$ of $\ph$ at the singular point of $\XFkpm$.

\begin{figure}[h]
  \begin{center}
	\includegraphics[scale=1]{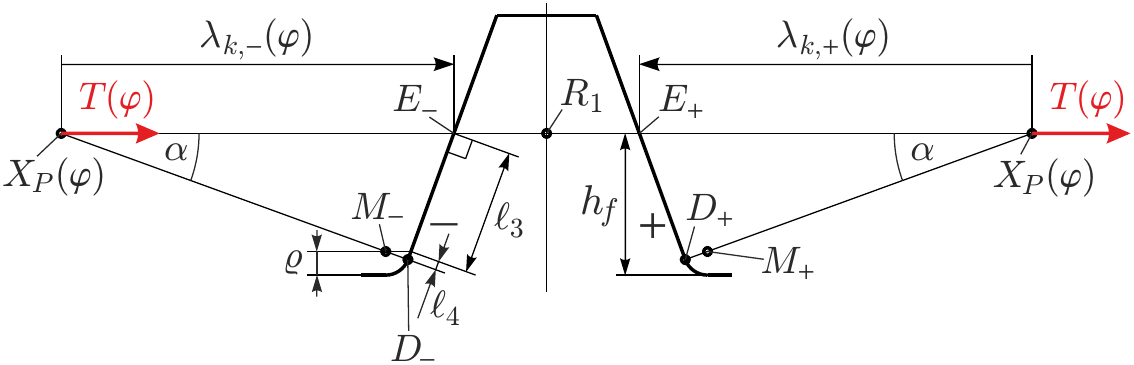}
	\caption{Sketch for the situation when the contact point of $\XFkpm$ and $\Xrkpm$ is generated}
 	\label{Fig:Contact_point02}   
  \end{center}
\end{figure}

Now, we determine a condition for the value of $\ph$ at the contact point of $\XFkpm$ and $\Xrkpm$ (see Fig.\ \ref{Fig:Curves_for_tooth_k}).
We denote by $D_\pm$ the point where the rack-cutter ``$\pm$''-flank is tangent to its rack-cutter fillet (see Fig.\ \ref{Fig:Contact_point02}). 
When the contact point is generated, the rack-cutter ``$\pm$''-flank at its point $D_\pm$ is tangent to $\XFkpm$ (not shown), and at the same time to $\Xrkpm$ (also not shown). 
So the rectangular triangle $D_\pm\,E_\pm\,X_P(\ph)$ gives
\beq
  \sin\alpha
= \frac{\ell_3+\ell_4}{\mp\lak(\ph)}\,,  
\eeq
where, for $\ell_4$ see also Fig.\ \ref{Fig:Reference_profile01}, 
\beq
  \ell_3 = (\hf-\rh)\sec\alpha\,,\qquad
  \ell_4 = \rh\tan\alpha\,,
\eeq
hence
\beq
  \mp\lak(\ph)
= \frac{\hf-\rh}{\cos\alpha \sin\alpha} + \frac{\rh}{\cos\alpha}\,,
\eeq
and therefore
\beqn \label{Eq:-+lambda*cos(alpha)}
  {\mp}\lak(\ph)\cos\alpha
= (\hf-\rh)\csc\alpha + \rh\,.    
\eeqn
This is the condition for the value $\phBkpm$ of the parameter $\ph$ at the contact point.
We write the condition \eqref{Eq:lambda*kappa} for the singular point as
\beqn \label{Eq:lambda*cos(alpha)_S}
  \lak(\ph)\cos\alpha
= \pm \frac{\sin\alpha}{\kappa(\ph)}\,.  
\eeqn
Since the functions $\lak(\ph)$ are strictly decreasing, 
and the non-undercutting conditions are
\beq
  \phSkm \le \phBkm
  \qquad\mbox{and}\qquad
  \phSkp \ge \phBkp\,, 
\eeq
we have
\beq
  \la_{k,-}(\phSkm) \ge \la_{k,-}(\phBkm)
  \qquad\mbox{and}\qquad
  \la_{k,+}(\phSkp) \le \la_{k,+}(\phBkp)\,.
\eeq
Therefore, from \eqref{Eq:-+lambda*cos(alpha)}, \eqref{Eq:lambda*cos(alpha)_S}, and $\cos\alpha > 0$, $0 \le \alpha < \pi/2$, one gets
\beq
  -\frac{\sin\alpha}{\kappa(\phSkm)} \ge (\hf-\rh)\csc\alpha + \rh
  \qquad\mbox{and}\qquad
  \frac{\sin\alpha}{\kappa(\phSkp)} \le -(\hf-\rh)\csc\alpha - \rh. 
\eeq
Putting together these inequalities, we have
\beq
  \frac{\sin\alpha}{-\kappa(\phSkpm)} \ge (\hf-\rh)\csc\alpha + \rh\,.
\eeq 
Since $\kappa(\ph) \le 0$, $0 \le \ph \le 2\pi$, we have $-\kappa(\phSkpm) \ge 0$ and
\beq
  -\kappa(\phSkpm)
\le \frac{\sin\alpha}{(\hf-\rh)\csc\alpha + \rh}
= \frac{\sin^2\alpha}{\hf - \rh + \rh\sin\alpha}\,.
\eeq
So we have found the  following theorem.

\begin{thm}
The \tn{``$\pm$''}-flank of tooth $k$, $k = 1,2,\ldots,z_1$, of the drive gear is free of undercut if  
\beqn \label{Eq:-kappa(ph_(S,pm))}
  {-}\kappa(\phSkpm)
\le \frac{\sin^2\alpha}{\hf - \rh(1-\sin\alpha)}\,,
\eeqn
where
\begin{itemize}[leftmargin=0.5cm]
\setlength{\itemsep}{-3pt}
\item $\kappa(\ph)$ is the curvature at point $X_P(\ph)$ of the centrode $X_P$ according to \eqref{Eq:kappa(phi)},
\item $\phSkpm$ the value of $\ph$ at the singular point (cusp) of $\XFkpm$ which is the solution of \eqref{Eq:lambda*kappa},
\item $\alpha$, $\hf$ and $\rh$ are the profile angle, the dedendum and the fillet radius, respectively {\em (}see Fig.\ {\em \ref{Fig:Reference_profile01})}.  
\end{itemize}
\end{thm}

Note that, besides $\alpha$, $\hf$ and $\rh$, the condition \eqref{Eq:-kappa(ph_(S,pm))} {\em depends only on the curvature $\kappa(\ph)$ of the centrode $X_P$ at the only point $X_P(\phSkpm)$}.
So, in order to check if tooth $k$ is free of undercut, it is necessary to determine the singular points $\phSkpm$ from \eqref{Eq:lambda*kappa} and then use \eqref{Eq:-kappa(ph_(S,pm))}.

\begin{coro}
Every tooth of the drive gear is free of undercut if
\beqn \label{Eq:-kappa(ph)}
  {-}\kappa(\ph)
\le \frac{\sin^2\alpha}{\hf - \rh(1-\sin\alpha)}
  \quad\mbox{for}\quad 0 \le \ph \le 2\pi\,.
\eeqn
\end{coro}

Introducing the dimensionless quantities
\beqn \label{Eq:hf^*_and_rho^*}
  \hf^* := \frac{\hf}{m}
  \qquad\mbox{and}\qquad
  \rh^* := \frac{\rh}{m}\,, 
\eeqn
from \eqref{Eq:-kappa(ph)} we get the condition 
\beqn \label{Eq:m <= ...}
  m
\le \frac{\sin^2\alpha}{\ds \left(\hf^*-\rh^*(1-\sin\alpha)\right)\max_{\,0\le\ph\le 2\pi}(-\kappa(\ph))}  
\eeqn
for the values of the module $m$ which provide that every tooth of the drive gear is free of undercut.
$\hf^*$ and $\rh^*$ in \eqref{Eq:m <= ...} can be freely chosen within reasonable limits, e.g.\ according to the rack-cutters available or to \textcite{DIN3972}.
Note that $\kappa(\ph)$ depends on the pivot distance $a = \overline{O_1 O_2}$.
A greater value of~$a$ yields a smaller value of $\max_{\,0\le\ph\le 2\pi}(-\kappa(\ph))$; hence greater values of $m$ are possible.
On the other hand, the chosen value of $m$ influences $a$.
From \eqref{Eq:m <= ...} with \eqref{Eq:kappa(phi)} one gets
\beq
  \frac{m}{a}
\le \frac
	{\sin^2\alpha}
	{\ds \big(\hf^*-\rh^*(1-\sin\alpha)\big)\,\max_{\,0\le\ph\le 2\pi}
		\left|\frac{(1+\psi'(\ph))^2\, h(\ph)}{w(\ph)}\right|}\,,  
\eeq
hence
\beqn \label{Eq:a/m >= ...}
  \frac{a}{m}
\ge \frac{\hf^*-\rh^*(1-\sin\alpha)}{\sin^2\alpha}\,
  \max_{\,0\le\ph\le 2\pi}\left|\frac{(1+\psi'(\ph))^2\, h(\ph)}{w(\ph)}\right|.  
\eeqn
Inequality \eqref{Eq:a/m >= ...} provides an useful criterion to determine pivot distance $a$ and module $m$ if undercut is not allowed.
The ratio $a/m$ can be considered as {\em dimensionless pivot distance} in a certain sense.
(Clearly, for $m = 1\,\tn{mm}$ it is the real pivot distance without the unit ``\tn{mm}''.)
With \eqref{Eq:a} we get
\beq
  z_1
\ge \frac{I(0,2\pi)\,\big(\hf^*-\rh^*(1-\sin\alpha)\big)}{\pi\sin^2\alpha}\,
  \max_{\,0\le\ph\le 2\pi}\left|\frac{(1+\psi'(\ph))^2\, h(\ph)}{w(\ph)}\right|  
\eeq
as criterion for the number $z_1$ of teeth necessary to ensure that each tooth is free of undercut.

\subsection{Driven gear (undercut)}

We determine the singular point $\phXiSkpm$ of $\XiFkpm$.
It follows from the condition
\beqn \label{Eq:XiFkpm'(phi)=0}
  \XiFkpm'(\phXiSkpm)=0\,.
\eeqn
Applying \eqref{Eq:TXi}, \eqref{Eq:TXi'} and \eqref{Eq:lambda_k'=-|Xi_P'|}, we have
\begin{align*}
  \XiFkpm'(\ph)
= {} & \varXi_P'(\ph) + \left(\lak'(\ph)\,\TXi(\ph) + \lak(\ph)\,\TXi'(\ph)\right)
  \ee^{\pm\ii\alpha}\cos\alpha\\[0.05cm]
= {} & |\varXi_P'(\ph)|\,\TXi(\ph) + \left(-|\varXi_P'(\ph)| + \lak(\ph)\,\ii \hXi(\ph)\right)
  \TXi(\ph)\,\ee^{\pm\ii\alpha}\cos\alpha\\[0.05cm] 
= {} & \left[|\varXi_P'(\ph)| + \left(-|\varXi_P'(\ph)| + \lak(\ph)\,\ii \hXi(\ph)\right)
  (\cos\alpha \pm \ii\sin\alpha)\cos\alpha\right] \TXi(\ph)\\[0.05cm]
= {} & \left[\left(|\varXi_P'(\ph)|\sin\alpha \mp \lak(\ph)\,\hXi(\ph)\cos\alpha\right)\sin\alpha\right.\\
& \:\, + \left.\ii\left(\lak(\ph)\,\hXi(\ph)\cos\alpha \mp |\varXi_P'(\ph)|\sin\alpha\right)\cos\alpha\right]
  \TXi(\ph)\,.
\end{align*}
From \eqref{Eq:XiFkpm'(phi)=0} it follows that the term in the square brackets is equal to zero, and so real and imaginary part of this term are equal to zero which gives
\beq
  \lak(\ph)\,\hXi(\ph)\cos\alpha \mp |\varXi_P'(\ph)|\sin\alpha = 0\,, 
\eeq
hence
\beq
  \lak(\ph)\,\hXi(\ph) = \pm |\varXi_P'(\ph)|\tan\alpha\,.
\eeq
Using \eqref{Eq:hXi=s'*kaXi} and \eqref{Eq:s'(phi)=|Xi_P'(phi)|}, we get
\beqn \label{Eq:lambda*kaXi}
  \lak(\ph)\,\kaXi(\ph)
= \pm\tan\alpha
\eeqn
as condition for the values $\phXiSkpm$ of $\ph$ at the singular points of $\XiFkpm$.
One finds
\beqn \label{Eq:+-lambda*cos(alpha)}
  {\pm}\lak(\ph)\cos\alpha
= (\hf-\rh)\csc\alpha + \rh    
\eeqn
as condition for the value $\phXiBkpm$ of $\ph$ at the contact point of $\XiFkpm$ and $\Xirkpm$.
Finally, one gets the following theorem.

\begin{thm}
The \tn{``$\pm$''}-flank of tooth space $k$, $k = 1,2,\ldots,z_2$, of the driven gear is free of undercut if  
\beqn \label{Eq:kaXi(ph_(S,pm))}
  \kaXi(\phXiSkpm)
\le \frac{\sin^2\alpha}{\hf - \rh(1-\sin\alpha)}\,,
\eeqn
where
\begin{itemize}[leftmargin=0.5cm]
\setlength{\itemsep}{-3pt}
\item $\kaXi(\ph)$ is the curvature at point $\varXi_P(\ph)$ of the centrode $\varXi_P$ according to \eqref{Eq:kaXi(phi)},
\item $\phXiSkpm$ the value of $\ph$ at the singular point (cusp) of $\XiFkpm$ which is the solution of \eqref{Eq:lambda*kaXi},
\item $\alpha$, $\hf$ and $\rh$ are the profile angle, the dedendum and the fillet radius, respectively {\em (}see Fig.\ {\em \ref{Fig:Reference_profile01})}.  
\end{itemize}
\end{thm}

\begin{coro}
Every tooth of the driven gear is free of undercut if
\beqn \label{Eq:kaXi(ph)}
  \kaXi(\ph)
\le \frac{\sin^2\alpha}{\hf - \rh(1-\sin\alpha)}
  \quad\mbox{for}\quad 0 \le \ph \le 2\pi\,.
\eeqn
\end{coro}


\section{The base curves}
\label{Sec:The_base_curves}

Unidirectional flanks of different teeth of a gear, i.e.\ all ``$+$''-flanks or all ``$-$''-flanks, can be generated as involutes of a so-called {\em base curve} \cite[p.\ 241]{Wunderlich}.
The respective base curve is the evolute of all unidirectional flanks.

\subsection{Base curve of the drive gear}

The following theorem establishes a parametric equation for the base curves $\XBpm$ of the drive gear.

\begin{lem} \label{Lem:X_B}
A parametric equation of the base curves $\XBpm$ of the drive gear is given by
\beqn \label{Eq:X_B}
  \XBpm(\ph) = X_P(\ph) \pm \frac{\sin\alpha}{\kappa(\ph)}\,T(\ph)\,\ee^{\pm\ii\alpha}\,,\quad
  0 \le \ph \le 2\pi\,, 
\eeqn
where $\kappa(\ph)$ and $T(\ph)$ are the curvature and the tangent unit vector, respectively, of the centrode $X_P$ at point $\ph$. 
\end{lem}

\begin{proof}
The base curve is obtained as the envelope of all straight lines that intersect the centrode at the same angle $\alpha$ \cite[p.\ 241]{Wunderlich}.
The two straight lines that intersect the centrode $X_P$ at point~$\ph$ with the angles $-\alpha$ and $+\alpha$, respectively, are given by
\beqn \label{Eq:LOA}
  X = X_P(\ph) + \mu\,T(\ph)\,\ee^{\pm\ii\alpha}\,,\quad
  \mu \in \R\,. 
\eeqn 
Now we use \eqref{Eq:X-1}, and choose, with $\mu = 1$ and $ \mu = 0$, respectively, in \eqref{Eq:LOA},
\begin{align*}
  A = A(\ph)
:= {} & X_P(\ph) + T(\ph)\,\ee^{\pm\ii\alpha}\,,\\[0.05cm] 
  B = B(\ph)
:= {} & X_P(\ph)\,.
\end{align*}
We have
\beq
  A(\ph) - B(\ph)
= T(\ph)\,\ee^{\pm\ii\alpha}  
\eeq
and, see \eqref{Eq:T'-2},
\beq
  (A(\ph) - B(\ph))'
= T'(\ph)\,\ee^{\pm\ii\alpha}  
= \ii\,h(\ph)\,T(\ph)\,\ee^{\pm\ii\alpha}\,.
\eeq
With this one finds
\begin{align*}
  [A-B,(A-B)']
= {} & [T\,\ee^{\pm\ii\alpha},\ii\,h\,T\,\ee^{\pm\ii\alpha}]  
= h\,[T\,\ee^{\pm\ii\alpha},\ii\,T\,\ee^{\pm\ii\alpha}]
= h\,T\,\ee^{\pm\ii\alpha}\,\overline{T\,\ee^{\pm\ii\alpha}}\,[1,\ii]\\[0.05cm]
= {} & h\,T\,\overline{T}\,\ee^{\pm\ii\alpha}\,\ee^{\mp\ii\alpha}\,[1,\ii]
= h\,.
\end{align*}
One also gets
\beq
  [A,B]
= [X_P+T\,\ee^{\pm\ii\alpha},X_P]  
= [X_P,X_P] + [T\,\ee^{\pm\ii\alpha},X_P]
= [T\,\ee^{\pm\ii\alpha},X_P]
\eeq
and
\begin{align*}
  [A,B]'
= {} & [T\,\ee^{\pm\ii\alpha},X_P]'
= [T'\,\ee^{\pm\ii\alpha},X_P] + [T\,\ee^{\pm\ii\alpha},X_P']
= [\ii\,h\,T\,\ee^{\pm\ii\alpha},X_P] + [T\,\ee^{\pm\ii\alpha},|X_P'|\,T]\\[0.05cm]
= {} & h\,[\ii\,T\,\ee^{\pm\ii\alpha},X_P] + |X_P'|\,[T\,\ee^{\pm\ii\alpha},T]  
= h\,[\ii\,T\,\ee^{\pm\ii\alpha},X_P] + |X_P'|\,T\,\overline{T}\,[\ee^{\pm\ii\alpha},1]\\[0.05cm]
= {} & h\,[\ii\,T\,\ee^{\pm\ii\alpha},X_P] - |X_P'|\,[1,\ee^{\pm\ii\alpha}]
= h\,[\ii\,T\,\ee^{\pm\ii\alpha},X_P] - |X_P'|\,\Imz\ee^{\pm\ii\alpha}\\[0.05cm]
= {} & h\,[\ii\,T\,\ee^{\pm\ii\alpha},X_P] \mp |X_P'|\sin\alpha\,.
\end{align*}
Thus we obtain by means of \eqref{Eq:X-1}
\begin{align*}
  X
= {} & \frac
	{[T\,\ee^{\pm\ii\alpha},X_P]\,\ii\,h\,T\,\ee^{\pm\ii\alpha}
	- \left(h\,[\ii\,T\,\ee^{\pm\ii\alpha},X_P] \mp |X_P'|\sin\alpha\right)T\,\ee^{\pm\ii\alpha}}
	{h}\\[0.05cm]  
= {} & \frac
	{\left([T\,\ee^{\pm\ii\alpha},X_P]\,\ii - [\ii\,T\,\ee^{\pm\ii\alpha},X_P]\right)h \pm |X_P'|\sin\alpha}
	{h}\,T\,\ee^{\pm\ii\alpha}\\[0.05cm]
= {} & \left([T\,\ee^{\pm\ii\alpha},X_P]\,\ii - [\ii\,T\,\ee^{\pm\ii\alpha},X_P]\right)T\,\ee^{\pm\ii\alpha}
  \pm \frac{1}{h}\,|X_P'|\,T\,\ee^{\pm\ii\alpha}\sin\alpha\,.	
\end{align*}
Now we get
\begin{align*}
& T\,\ee^{\pm\ii\alpha},X_P]\,\ii - [\ii\,T\,\ee^{\pm\ii\alpha},X_P]\\[0.05cm]
& = {} {-}\frac{1}{2}\left(T\,\ee^{\pm\ii\alpha}\,\overline{X}_P
		- \overline{T}\,\ee^{\mp\ii\alpha}\,X_P\right)
  - \frac{\ii}{2}\left(\ii\,T\,\ee^{\pm\ii\alpha}\,\overline{X}_P
		+ \ii\,\overline{T}\,\ee^{\mp\ii\alpha}\,X_P\right)\\[0.05cm]
& = {} {-}\frac{1}{2}\left(T\,\ee^{\pm\ii\alpha}\,\overline{X}_P
		- \overline{T}\,\ee^{\mp\ii\alpha}\,X_P\right)
  + \frac{1}{2}\left(\,T\,\ee^{\pm\ii\alpha}\,\overline{X}_P 
		+ \overline{T}\,\ee^{\mp\ii\alpha}\,X_P\right)\\[0.05cm]
& = {} \overline{T}\,\ee^{\mp\ii\alpha}\,X_P\,, 				    
\end{align*}
hence
\beq
  X
= X_P \pm \frac{1}{h}\,|X_P'|\,T\,\ee^{\pm\ii\alpha}\sin\alpha\,.  
\eeq
With \eqref{Eq:s'} and \eqref{Eq:kappa(phi)} we have
\beqn \label{Eq:|X_P'|-2}
  |X_P'(\ph)| = \frac{h(\ph)}{\kappa(\ph)}\,,
\eeqn
and so finally follows
\beq
  X = X_P(\ph) \pm \frac{\sin\alpha}{\kappa(\ph)}\,T(\ph)\,\ee^{\pm\ii\alpha}\,. \qedhere
\eeq
\end{proof}

Our next aim is to determine the equation of the tooth flanks as involutes of the base curves.
We start with the determination of the arc length element of the respective base curve $\XBpm$.
Differentiation of \eqref{Eq:X_B} with respect to $\ph$ under consideration of \eqref{Eq:T'-2}, \eqref{Eq:|X_P'|-2} and \eqref{Eq:T} yields
\begin{align*}
  \XBpm'(\ph)
= {} & X_P'(\ph) \pm \ee^{\pm\ii\alpha}\sin\alpha
  \left(\frac{T'(\ph)}{\kappa(\ph)} - \frac{\kappa'(\ph)\,T(\ph)}{\kappa^2(\ph)}\right)\\[0.05cm]
= {} & X_P'(\ph) \pm \ee^{\pm\ii\alpha}\sin\alpha
  \left(\frac{\ii\,h(\ph)\,T(\ph)}{\kappa(\ph)} - \frac{\kappa'(\ph)\,T(\ph)}{\kappa^2(\ph)}\right)\\[0.05cm]
= {} & \left[|X_P'(\ph)| \pm \ee^{\pm\ii\alpha}\sin\alpha
  \left(\ii\,|X_P'(\ph)| - \frac{\kappa'(\ph)}{\kappa^2(\ph)}\right)\right] T(\ph)
  \displaybreak[0]\\[0.05cm]
= {} & \left[|X_P'(\ph)| \pm (\cos\alpha \pm \ii\sin\alpha)\sin\alpha
  \left(\ii\,|X_P'(\ph)| - \frac{\kappa'(\ph)}{\kappa^2(\ph)}\right)\right] T(\ph)
  \displaybreak[0]\\[0.05cm]
= {} & \left[|X_P'(\ph)| \mp (\cos\alpha \pm \ii\sin\alpha)
  \left(\frac{\kappa'(\ph)}{\kappa^2(\ph)} - \ii\,|X_P'(\ph)|\right)\sin\alpha\right] T(\ph)
  \displaybreak[0]\\[0.05cm]
= {} & \left[|X_P'(\ph)|
  \mp \left(\frac{\kappa'(\ph)}{\kappa^2(\ph)}\,\cos\alpha \pm |X_P'(\ph)|\sin\alpha\right)\sin\alpha\right.\\[0.05cm]
& \hspace{1.4cm} \pm \left.\ii\left(|X_P'(\ph)|\cos\alpha \mp \frac{\kappa'(\ph)}{\kappa^2(\ph)}\,\sin\alpha\right)\sin\alpha\right] T(\ph)\\[0.05cm]
= {} & \left[|X_P'(\ph)|
  \mp \frac{\kappa'(\ph)}{\kappa^2(\ph)}\cos\alpha\sin\alpha - |X_P'(\ph)|\sin^2\alpha\right.\\[0.05cm]
& \hspace{1.4cm} \pm \left.\ii\left(|X_P'(\ph)|\cos\alpha \mp \frac{\kappa'(\ph)}{\kappa^2(\ph)}\,\sin\alpha\right)\sin\alpha\right] T(\ph)
  \displaybreak[0]\\[0.05cm]
= {} & \left[|X_P'(\ph)|\cos^2\alpha
  \mp \frac{\kappa'(\ph)}{\kappa^2(\ph)}\cos\alpha\sin\alpha\right.\\[0.05cm]
& \hspace{1.4cm} \pm \left.\ii\left(|X_P'(\ph)|\cos\alpha \mp \frac{\kappa'(\ph)}{\kappa^2(\ph)}\,\sin\alpha\right)\sin\alpha\right] T(\ph)
  \displaybreak[0]\\[0.05cm]
= {} & \left[\left(|X_P'(\ph)|\cos\alpha \mp \frac{\kappa'(\ph)}{\kappa^2(\ph)}\sin\alpha\right)\cos\alpha\right.\\[0.05cm]
& \hspace{1.4cm}  \pm \left.\ii\left(|X_P'(\ph)|\cos\alpha \mp \frac{\kappa'(\ph)}{\kappa^2(\ph)}\,\sin\alpha\right)\sin\alpha\right] T(\ph)
  \displaybreak[0]\\[0.05cm]
= {} & \left(|X_P'(\ph)|\cos\alpha \mp \frac{\kappa'(\ph)}{\kappa^2(\ph)}\,\sin\alpha\right)\ee^{\pm\ii\alpha}\,T(\ph)\,,
\end{align*}
hence
\beq
  |\XBpm'(\ph)|
= \left||X_P'(\ph)|\cos\alpha \mp \frac{\kappa'(\ph)}{\kappa^2(\ph)}\,\sin\alpha\right|.  
\eeq
Since we do not want to calculate the length of the curve $\XBpm$, but the length of its developments, it is convenient to use
\beq
  \left(|X_P'(\ph)|\cos\alpha \mp \frac{\kappa'(\ph)}{\kappa^2(\ph)}\,\sin\alpha\right)\dd\ph
\eeq
as length element.
This will simplify the calculations.
In addition, one avoids the case distinctions that become necessary if $\XBpm$ has cusps (see Fig.\ \ref{Fig:Evolute_and_involutes01}).
We denote by $\XFp^*$ the involute of the base curve $\XBp$ that we obtain when the singular point (cusp) $\phSp$ of $\XFp^*$ is freely chosen in the interval $0 \le \ph \le 2\pi$.
Using \eqref{Eq:|X_P'|-1} and \eqref{Eq:I}, one gets
\begin{align} \label{Eq:XFp^*}
  \XFp^*(\ph)
= {} & \XBp(\ph) + \left[
  \int_\ph^{\phSp}\left(|X_P'(\ph)|\cos\alpha - \frac{\kappa'(\ph)}{\kappa^2(\ph)}\,\sin\alpha\right)\dd\ph
  \right] T(\ph)\,\ee^{\ii\alpha}\nonumber\displaybreak[0]\\[0.05cm]
= {} & \XBp(\ph) + \left[
  a\cos\alpha\int_\ph^{\phSp}\frac{w(\ph)}{(1+\psi'(\ph))^2}\,\dd\ph
  + \sin\alpha\int_\ph^{\phSp}\left(-\frac{\kappa'(\ph)}{\kappa^2(\ph)}\right)\dd\ph
  \right] T(\ph)\,\ee^{\ii\alpha}\nonumber\displaybreak[0]\\[0.05cm]
= {} & \XBp(\ph) + \left[
  a\cos\alpha\,I(\ph,\phSp)
  + \sin\alpha\left(\frac{1}{\kappa(\phSp)}-\frac{1}{\kappa(\ph)}\right)
  \right] T(\ph)\,\ee^{\ii\alpha}\nonumber\\[0.05cm]
= {} & \XBp(\ph) - \left[
  a\cos\alpha\,I(\phSp,\ph)
  - \sin\alpha\left(\frac{1}{\kappa(\phSp)}-\frac{1}{\kappa(\ph)}\right)
  \right] T(\ph)\,\ee^{\ii\alpha}\,.  
\end{align}
For the involutes $\XFm^*$ of the base curve $\XBm$ we have
\begin{align} \label{Eq:XFm^*}
  \XFm^*(\ph)
= {} & \XBm(\ph) - \left[
  \int_{\phSm}^\ph\left(|X_P'(\ph)|\cos\alpha + \frac{\kappa'(\ph)}{\kappa^2(\ph)}\,\sin\alpha\right)\dd\ph
  \right] T(\ph)\,\ee^{-\ii\alpha}\displaybreak[0]\nonumber\\[0.05cm]
= {} & \XBm(\ph) - \left[
  a\cos\alpha\,I(\phSm,\ph) + \sin\alpha\left(\frac{1}{\kappa(\phSm)}-\frac{1}{\kappa(\ph)}\right)
  \right] T(\ph)\,\ee^{-\ii\alpha}\,,      
\end{align}
where $\phSm$ is the singular point (cusp) of $\XFm^*$.
So writing \eqref{Eq:XFp^*} and \eqref{Eq:XFm^*} together, we have
\beqn \label{Eq:XFpm^*}
  \XFpm^*(\ph)
= \XBpm(\ph) - \left[
  a\cos\alpha\,I(\phSpm,\ph) \mp \sin\alpha\left(\frac{1}{\kappa(\phSpm)}-\frac{1}{\kappa(\ph)}\right)
  \right] T(\ph)\,\ee^{\pm\ii\alpha}\,.  
\eeqn
Applying Lemma \ref{Lem:X_B} yields
\begin{align*}
  \XFpm^*(\ph)
= {} & X_P(\ph) \pm \frac{\sin\alpha}{\kappa(\ph)}\,T(\ph)\,\ee^{\pm\ii\alpha}\nonumber\\
& - \bigg[a\cos\alpha\,I(\phSpm,\ph)
  \mp \left.\sin\alpha\left(\frac{1}{\kappa(\phSpm)}-\frac{1}{\kappa(\ph)}\right)
  \right] T(\ph)\,\ee^{\pm\ii\alpha}\nonumber\\[0.05cm]
= {} & X_P(\ph) - \left(a\cos\alpha\,I(\phSpm,\ph)
  \mp \frac{\sin\alpha}{\kappa(\phSpm)}\right) T(\ph)\,\ee^{\pm\ii\alpha}\,.  
\end{align*}
We formulate our result as theorem.

\begin{thm} \label{Thm:XFpm^*}
A parametric equation of the involutes $\XFpm^*$ of the base curve $\XBpm$ is given by
\beq
  \XFpm^*(\ph)
= X_P(\ph) - \left(a\cos\alpha\,I(\phSpm,\ph)
  \mp \frac{\sin\alpha}{\kappa(\phSpm)}\right) T(\ph)\,\ee^{\pm\ii\alpha}\,,  
\eeq
where $\phSpm$ is the value of $\ph$ at the singular point of $\XFpm^*$, $I$ is the integral \eqref{Eq:I}, and $\kappa$ and $T$ are the curvature \eqref{Eq:kappa(phi)} and the tangent unit vector \eqref{Eq:T}, respectively, of the centrode $X_P$.  
\end{thm}

It should be emphasized that $\kappa(\phSpm)$ in Theorem \ref{Thm:XFpm^*} is the curvature of the {\bf centrode} $\boldsymbol{X_P}$ at the singular point $\phSpm$ of the {\bf involute} $\boldsymbol{\XFpm^*}$.   

It remains to be shown that the involutes obtained in this way are indeed the tooth flanks according to Theorem~\ref{Thm:X_F} (cf.\ the quote at the beginning of this section).
This will be proved in the following corollary. 

\begin{coro}
The flank curve $\XFkpm$, $k \in \{1,\ldots,z_1\}$, generated with the rack-cutter is identical to the involute $\XFkpm^*$ of the base curve $\XBpm$ whose parametric equation is
\beq
  \XFkpm^*(\ph)
= X_P(\ph) - \left(a\cos\alpha\,I(\phSkpm,\ph)
  \mp \frac{\sin\alpha}{\kappa(\phSkpm)}\right) T(\ph)\,\ee^{\pm\ii\alpha}\,.  
\eeq
\end{coro}

\begin{proof}
The midpoint of tooth $k$ on the centrode $X_P$ is given by $X_P(\chi(k))$ with $\chi(k)$ according to \eqref{Eq:chi(k)}.
Condition \eqref{Eq:lambda*cos(alpha)_S} for the value $\phSkpm$ of $\ph$ at the singular point of the flank $\XFkpm$ states that
\beq
  \lak(\phSkpm)\cos\alpha
= \pm \frac{\sin\alpha}{\kappa(\phSkpm)}\,.  
\eeq
Using \eqref{Eq:lambda_k}, it follows that
\begin{align*}
  \XFkpm^*(\ph)
= {} & X_P(\ph) - \left(a\cos\alpha\,I(\phSkpm,\ph)
  \mp \frac{\sin\alpha}{\kappa(\phSkpm)}\right) T(\ph)\,\ee^{\pm\ii\alpha}\\[0.05cm] 
= {} & X_P(\ph) - \big[a\cos\alpha\,I(\phSkpm,\ph)
  \mp \big({\pm}\lak(\phSkpm)\cos\alpha\big)\big]\, T(\ph)\,\ee^{\pm\ii\alpha}\\[0.05cm]  
= {} & X_P(\ph) - \big[aI(\phSkpm,\ph)
  - \lak(\phSkpm)\big]\, T(\ph)\,\ee^{\pm\ii\alpha}\cos\alpha\\[0.05cm]
= {} & X_P(\ph) - \left[aI(\phSkpm,\ph)
  \mp \frac{\pi m}{4} + aI(\chi(k),\phSkpm)\right] T(\ph)\,\ee^{\pm\ii\alpha}\cos\alpha\\[0.05cm]
= {} & X_P(\ph) - \left[\mp \frac{\pi m}{4} + a\,\big(I(\chi(k),\phSkpm) + I(\phSkpm,\ph)\big)\right]
  T(\ph)\,\ee^{\pm\ii\alpha}\cos\alpha\\[0.05cm]
= {} & X_P(\ph) + \left(\pm \frac{\pi m}{4} - aI(\chi(k),\ph)\right)
  T(\ph)\,\ee^{\pm\ii\alpha}\cos\alpha\\[0.05cm]
= {} & X_P(\ph) + \lak(\ph)\,T(\ph)\,\ee^{\pm\ii\alpha}\cos\alpha\,.  
\end{align*}
One sees that the last line is identical to the parametric equation $\XFkpm(\ph)$ of $\XFkpm$ in Theorem~\ref{Thm:X_F}.
\end{proof}

The following corollary is important for the calculation, or at least estimation, of the tooth flank pressure by means of the Hertzian theory.

\begin{coro} \label{Coro:Curvature_radius}
The radius of curvature at point $\ph$ of the involute $\XFkpm^*\equiv\XFkpm$ is given by
\beq
  \left|a\cos\alpha\,I(\phSkpm) \mp \sin\alpha 
  \left(\frac{1}{\kappa(\phSkpm,\ph)} - \frac{1}{\kappa(\ph)}\right)
  \right|. 
\eeq
\end{coro}

\begin{proof}
Since $\XBpm$ is the evolute of the involutes $\XFkpm^*$, from \eqref{Eq:XFpm^*} it is known that
\beq
  -\left[a\cos\alpha\,I(\phSkpm,\ph)
  \mp \sin\alpha\left(\frac{1}{\kappa(\phSkpm)} \mp \frac{1}{\kappa(\ph)}\right)\right]
\eeq
is the signed radius of curvature at point $\XFkpm^*(\ph)$ of $\XFkpm^*$.
The proposition of Corollary \ref{Coro:Curvature_radius} follows immediately. 
\end{proof}

\subsection{Base curve of the driven gear}

\begin{lem} \label{Lem:Xi_B}
A parametric equation of the base curves $\XiBpm$ of the driven gear is given by
\beqn \label{Eq:Xi_B}
  \XiBpm(\ph) = \varXi_P(\ph) \pm \frac{\sin\alpha}{\kaXi(\ph)}\,\TXi(\ph)\,\ee^{\pm\ii\alpha}\,,\quad
  0 \le \ph \le 2\pi\,, 
\eeqn
where $\kaXi(\ph)$ and $\TXi(\ph)$ are the curvature and the tangent unit vector, respectively, of the centrode $\varXi_P$ at point $\ph$. 
\end{lem}

\begin{proof}
The two straight lines that intersect the centrode $\varXi_P$ at point~$\ph$ with the angles $-\alpha$ and $+\alpha$, respectively, are given by
\beq
  \varXi = \varXi_P(\ph) + \mu\,\TXi(\ph)\,\ee^{\pm\ii\alpha}\,,\quad
  \mu \in \R\,. 
\eeq
The proof is now analogous to that of Lemma \ref{Lem:X_B}.
\end{proof}

\begin{thm} \label{Thm:XiFpm^*}
A parametric equation of the involutes $\XiFpm^*$ of the base curve $\XiBpm$ is given by
\beq
  \XiFpm^*(\ph)
= \varXi_P(\ph) - \left(a\cos\alpha\,I(\phSpm,\ph)
  \mp \frac{\sin\alpha}{\kaXi(\phSpm)}\right) \TXi(\ph)\,\ee^{\pm\ii\alpha}\,,  
\eeq
where $\phSpm$ is value of $\ph$ at the singular point of $\XiFpm^*$, $I$ is the integral \eqref{Eq:I}, and $\kaXi$ and $\TXi$ are the curvature \eqref{Eq:kaXi(phi)} and the tangent unit vector \eqref{Eq:TXi}, respectively, of the centrode $\varXi_P$.
\end{thm}

\begin{proof}
The proof is analogous to that of Theorem \ref{Thm:XFpm^*}.
\end{proof}

\begin{coro}
The flank curve $\XiFkpm$, $k \in \{1,\ldots,z_2\}$, generated with the rack-cutter is identical to the involute $\XiFkpm^*$ of the base curve $\XiBpm$ whose parametric equation is
\beq
  \XiFkpm^*(\ph)
= \varXi_P(\ph) - \left(a\cos\alpha\,I(\phSkpm,\ph)
  \mp \frac{\sin\alpha}{\kaXi(\phXiSkpm)}\right) \TXi(\ph)\,\ee^{\pm\ii\alpha}\,,  
\eeq
where $\phXiSkpm$ is the value of $\ph$ at the singular point of $\XiFkpm$.
\end{coro}

\begin{proof}
The midpoint of tooth $k$ on the centrode $\varXi_P$ is given by $\varXi_P(\chi(k))$ with $\chi(k)$ according to \eqref{Eq:chi(k)}.
Condition \eqref{Eq:lambda*kaXi} for the value $\phXiSkpm$ of $\ph$ at the singular point of the flank $\XiFkpm$ states that
\beq
  \lak(\phXiSkpm)\cos\alpha
= \pm \frac{\sin\alpha}{\kaXi(\phXiSkpm)}\,.  
\eeq
Using \eqref{Eq:lambda_k}, it follows that
\begin{align*}
  \XiFkpm^*(\ph)
= {} & \varXi_P(\ph) - \left(a\cos\alpha\,I(\phXiSkpm,\ph)
  \mp \frac{\sin\alpha}{\kaXi(\phXiSkpm)}\right) \TXi(\ph)\,\ee^{\pm\ii\alpha}\\[0.05cm] 
= {} & \varXi_P(\ph) - \big[a\cos\alpha\,I(\phXiSkpm,\ph)
  \mp \big({\pm}\lak(\phXiSkpm)\cos\alpha\big)\big]\, \TXi(\ph)\,\ee^{\pm\ii\alpha}\\[0.05cm]  
= {} & \varXi_P(\ph) - \big[aI(\phXiSkpm,\ph)
  - \lak(\phXiSkpm)\big]\, \TXi(\ph)\,\ee^{\pm\ii\alpha}\cos\alpha\\[0.05cm]
= {} & \varXi_P(\ph) - \left[aI(\phXiSkpm,\ph)
  \mp \frac{\pi m}{4} + aI(\chi(k),\phXiSkpm)\right] \TXi(\ph)\,\ee^{\pm\ii\alpha}\cos\alpha\\[0.05cm]
= {} & \varXi_P(\ph) - \left[\mp \frac{\pi m}{4} + a\,\big(I(\chi(k),\phXiSkpm) + I(\phXiSkpm,\ph)\big)\right]
  \TXi(\ph)\,\ee^{\pm\ii\alpha}\cos\alpha\\[0.05cm]
= {} & \varXi_P(\ph) + \left(\pm \frac{\pi m}{4} - aI(\chi(k),\ph)\right)
  \TXi(\ph)\,\ee^{\pm\ii\alpha}\cos\alpha\\[0.05cm]
= {} & \varXi_P(\ph) + \lak(\ph)\,\TXi(\ph)\,\ee^{\pm\ii\alpha}\cos\alpha\,.  
\end{align*}
One sees that the last line is identical to the parametric equation $\XiFkpm(\ph)$ of $\XiFkpm$ in Theorem~\ref{Thm:Xi_F}.
\end{proof}


\section{Algorithm}
\label{Sec:Algorithm}

{\bf For both gears:}
\begin{enumerate}[leftmargin=0.75cm]
\setlength{\itemsep}{-2pt}
\item Assume that the center distance $a$ is equal to 1, and check that the centrode $c_1$ of the drive gear and the centrode $c_2$ of the driven gear are both convex curves.
\item Evaluate the integral $I(0,2\pi)$ (see \eqref{Eq:I}); if necessary, numerically.
\item Choose the module $m$ and the number $z_1$ of teeth of the drive gear, and calculate the center distance with \eqref{Eq:a}.
\item Choose angle $\alpha$, addendum $h_a$, dedendum $\hf$ and fillet radius $\rh$ (see Fig.\ \ref{Fig:Reference_profile01}).
\end{enumerate}

{\bf For the drive gear:} \mynobreakpar
\begin{enumerate}[leftmargin=0.75cm]
\setlength{\itemsep}{-2pt}
\setcounter{enumi}{4}
\item \label{Item:chi(k)}
Determine the values $\chi(k)$, $k = 1,2,\ldots,z_1$, of the parameter (drive angle) $\ph$ from \eqref{Eq:chi(k)}.
(These values divide $X_P$ into arcs of equal length $p = \pi m$. $\chi(k)$ is the value for the mid of tooth $k$.)
\item Determine the values $\phSkpm$, $k=1,2,\ldots,z_1$, of $\ph$ at the singular points of the involutes $\XFkpm \equiv \XFkpm^*$ with \eqref{Eq:lambda*kappa}.
\item Check each tooth flank for the occurrence of undercut with \eqref{Eq:-kappa(ph_(S,pm))}.
Alternatively, the parameters of the gear can be specified using inequality \eqref{Eq:-kappa(ph)} so that no undercut occurs at all.
\item \label{Step:--flank}
Perform the calculations of this step for each ``$-$''-flank.
 
If the flank has undercut, then determine the values $\phFkmE$ and $\phAkmE$ of the angle $\ph$ as solution of 
\beqn \label{Eq:XFkm=Xrkm}
  \XFkm(\phFkmE)
= \Xrkm(\phAkmE)
\eeqn
where $\XFkm(\ph)$ and $\Xrkm(\ph)$ are given by \eqref{Eq:XFkpm} and \eqref{Eq:Xrkpm}, respectively.
If the flank is free of undercut, then $\phFkmE = \phAkmE = \phBkm$ and the value $\phBkm$ of $\ph$ is obtained as solution of \eqref{Eq:-+lambda*cos(alpha)}.
 
Determine the value $\phAkmZ$ of the angle $\ph$ as solution of \eqref{Eq:-lambda_k}.
Determine the values $\phFkmZ$ and $\phpkpE$ of the angle $\ph$ as solution of 
\beqn \label{Eq:XFkm=Xha}
  \XFkm(\phFkmZ)
= X_a(\phpkpE)
\eeqn
with $X_a(\ph) := X_{p,h_a,+}(\ph)$ ($\Xpdpm(\ph)$ see \eqref{Eq:Xpdpm}).
The ``$-$''-flank of tooth $k$ is then given by
\beqn \label{Eq:--flank}
  X
= \left\{\!\!
  \begin{array}{lc}
	\Xrkm(\ph), & 
 		\phAkmE \le \ph \le \phAkmZ\,,\\[0.2cm]
 	\XFkm(\ph)\,, &
 		\phFkmE \le \ph \le \phFkmZ\,.	 
  \end{array}
  \right.
\eeqn
\item \label{Step:+-flank}
Perform the calculations of this step for each ``$+$''-flank.

If the flank has undercut, then determine the values $\phFkpZ$ and $\phAkpZ$ of the angle $\ph$ as solution of 
\beqn \label{Eq:XFkp=Xrkp}
  \XFkp(\phFkpZ)
= \Xrkp(\phAkpZ)
\eeqn
where $\XFkp(\ph)$ and $\Xrkp(\ph)$ are given by \eqref{Eq:XFkpm} and \eqref{Eq:Xrkpm}, respectively.
If the flank is free of undercut, then $\phFkpZ = \phAkpZ = \phBkp$ and the value $\phBkp$ of $\ph$ is obtained as solution of \eqref{Eq:-+lambda*cos(alpha)}.

Determine the values $\phFkpE$ and $\phpkpZ$ of the angle $\ph$ as solution of 
\beqn \label{Eq:XFkp=Xha}
  \XFkp(\phFkpE)
= X_a(\phpkpZ)\,.
\eeqn
Determine the value $\phAkpE$ of the angle $\ph$ as solution of \eqref{Eq:+lambda_k}.
The ``$+$''-flank of tooth $k$ is then given by
\beqn \label{Eq:+-flank}
  X
= \left\{\!\!
  \begin{array}{lc}
	\XFkp(\ph)\,, &
 		\phFkpE	\le \ph \le \phFkpZ\,,\\[0.2cm]
	\Xrkp(\ph), &
 		\phAkpE \le \ph \le \phAkpZ\,.	 
  \end{array}
  \right.
\eeqn
\item The arc of the addendum curve $X_a$ for tooth $k$, $k = 1,\ldots,z_1$, is given by
\beqn \label{Eq:arc_addendum_curve}
  X
= X_a(\ph)\,,\quad
  \phpkpE \le \ph \le \phpkpZ\,,    
\eeqn
with $\phpkpE$ from Step \ref{Step:--flank}.
\item The arc of the dedendum curve $X_f$ between tooth $k$, $k = 1,\ldots,z_1-1$, and tooth $k+1$ is given by
\beqn \label{Eq:arc_dedendum_curve_a}
  X
= X_f(\ph)\,,\quad
  \phAkpE \le \ph \le \ph_{A,\:\!k+1,\:\!-,\:\!2}\,.    
\eeqn
The arc of the dedendum curve $X_f$ between tooth $z_1$ and tooth $1$ is given by
\beqn \label{Eq:arc_dedendum_curve_b}
  X
= X_f(\ph)\,,\quad
  \ph_{A,\:\!z_1,\:\!+,\:\!1} \le \ph \le 2\pi + \ph_{A,\:\!1,\:\!-,\:\!2}\,.   
\eeqn
\end{enumerate}

{\bf For the driven gear:} \mynobreakpar
\begin{enumerate}[leftmargin=0.75cm]
\setlength{\itemsep}{-2pt}
\setcounter{enumi}{11}
\item Determine the values $\chi(k)$, $k = 1,2,\ldots,z_2$, of $\ph$ from \eqref{Eq:chi(k)}.
Since $\chi(1),\ldots,\chi(z_1)$ are already known from Step \ref{Item:chi(k)}, performing this step is only necessary to determine $\chi(z_1+1),\ldots,\chi(z_2)$ if $z_2 > z_1$.
\item Determine the values $\phXiSkpm$, $k=1,2,\ldots,z_2$, of $\ph$ at the singular points of the involutes $\XiFkpm \equiv \XiFkpm^*$ with \eqref{Eq:lambda*kaXi}.
\item Check each tooth flank for the occurrence of undercut with \eqref{Eq:kaXi(ph_(S,pm))}.
Alternatively, the parameters of the gear can be specified using inequality \eqref{Eq:kaXi(ph)} so that no undercut occurs at all.
\item \label{Step:--flankXi}
Perform the calculations of this step for each ``$-$''-flank.
 
If the flank has undercut, then determine the values $\phXiFkmZ$ and $\phXiAkmZ$ and of the angle $\ph$ as solution of 
\beqn \label{Eq:XiFkm=Xirkm}
  \XiFkm(\phXiFkmZ)
= \Xirkm(\phXiAkmZ)
\eeqn
where $\XiFkm(\ph)$ and $\Xirkm(\ph)$ are given by \eqref{Eq:XiFkpm} and \eqref{Eq:Xirkpm}, respectively.
If the flank is free of undercut, then $\phXiFkmZ = \phXiAkmZ = \phXiBkm$ and the value $\phXiBkm$ of $\ph$ is obtained as solution of \eqref{Eq:+-lambda*cos(alpha)}.
 
Determine the value $\phXiAkmE$ of the angle $\ph$ as solution of \eqref{Eq:+-lambda_k}.
Determine the values $\phXiFkmE$ and $\phXipkpE$ of the angle $\ph$ as solution of 
\beqn \label{Eq:XiFkm=Xiha}
  \XiFkm(\phXiFkmE)
= \varXi_a(\phXipkpE)
\eeqn
with $\varXi_a(\ph) \equiv \varXi_{p,h_a,+}(\ph)$ ($\Xipdpm(\ph)$ see \eqref{Eq:Xipdpm}).
The ``$-$''-flank of tooth space $k$ is then given by
\beqn \label{Eq:--flankXi}
  \varXi
= \left\{\!\!
  \begin{array}{lc}
	\XiFkm(\ph), & 
 		\phXiFkmE \le \ph \le \phXiFkmZ\,,\\[0.2cm]
 	\Xirkm(\ph)\,, &
 		\phXiAkmE \le \ph \le \phXiAkmZ\,.	 
  \end{array}
  \right.
\eeqn
\item \label{Step:+-flankXi}
Perform the calculations of this step for each ``$+$''-flank.

If the flank has undercut, then determine the values $\phXiFkpE$ and $\phXiAkpE$ of the angle $\ph$ as solution of 
\beqn \label{Eq:XiFkp=Xirkp}
  \XiFkp(\phXiFkpE)
= \Xirkp(\phXiAkpE)
\eeqn
where $\XiFkp(\ph)$ and $\Xirkp(\ph)$ are given by \eqref{Eq:XiFkpm} and \eqref{Eq:Xirkpm}, respectively.
If the flank is free of undercut, then $\phXiFkpE = \phXiAkpE = \phXiBkp$ and the value $\phXiBkp$ of $\ph$ is obtained as solution of \eqref{Eq:+-lambda*cos(alpha)}.

Determine the value $\phXiAkpZ$ of the angle $\ph$ as solution of \eqref{Eq:+-lambda_k}.
Determine the values $\phXiFkpZ$ and $\phXipkpZ$ of the angle $\ph$ as solution of 
\beqn \label{Eq:XiFkp=Xiha}
  \XiFkp(\phXiFkpZ)
= \varXi_a(\phXipkpZ)\,.
\eeqn
The ``$+$''-flank of tooth space $k$ is then given by
\beqn \label{Eq:+-flankXi}
  \varXi
= \left\{\!\!
  \begin{array}{lc}
	\XiFkp(\ph), & 
 		\phXiFkpE \le \ph \le \phXiFkpZ\,,\\[0.2cm]
 	\Xirkp(\ph)\,, &
 		\phXiAkpE \le \ph \le \phXiAkpZ\,.	 
  \end{array}
  \right.
\eeqn
\item The arc of the dedendum curve $\varXi_f$ of tooth space $k$, $k = 1,\ldots,z_2$, is given by
\beqn \label{Eq:arc_dedendum_curve_Xi} 
  \varXi
= \varXi_f(\ph)\,,\quad
  \phXiAkmE \le \ph \le \phXiAkpZ\,.    
\eeqn
\item The arc of the addendum curve $\varXi_a$ between tooth space $k$, $k = 1,\ldots,z_2-1$, and tooth space $k+1$ is given by
\beqn \label{Eq:arc_addendum_curve_a_Xi}
  \varXi
= \varXi_a(\ph)\,,\quad
  \phXipkpZ \le \ph \le \widetilde{\ph}_{p,\:\!k+1,\:\!+,\:\!1}\,.    
\eeqn
The arc of the addendum curve $\varXi_a$ between tooth space $z_2$ and tooth space $1$ is given by
\beqn \label{Eq:arc_addendum_curve_b_Xi}
  \varXi
= \varXi_a(\ph)\,,\quad
  \widetilde{\ph}_{p,\:\!z_2,\:\!+,\:\!2} \le \ph \le 2\pi + \widetilde{\ph}_{p,\:\!1,\:\!+,\:\!1}\,.   
\eeqn
\end{enumerate}

For the fast numerical solution of the occurring equations, e.g.\ \eqref{Eq:chi(k)} and \eqref{Eq:XFkm=Xrkm}, one can use the one-dimensional or two-dimensional Newton's method, respectively.  

\section{Example}
\label{Sec:Example}

Let the transmission funktion given by
\beq
  \psi(\ph)
= \ph - b\sin\ph\,,\quad 0 \le \ph \le 2\pi\,,\quad b = \tn{const} \in \R\,.  
\eeq
Note that $\psi(2\pi) = 2\pi$.

We consider the centrode $X_P$ of the drive gear.
From \eqref{Eq:r_and_R} it follows that $X_P$ has the polar equation    
\beq 
  r(\ph)
= \frac{1-b\cos\ph}{2-b\cos\ph}\,a\,.  
\eeq
We assume $a = 1$.
Considering Fig.\ \ref{Fig:Centrodes01}, one sees that not every value of $b$ results in a practically usable centrode~$X_P$.
Since the teeth are to be generated with a rack-cutter, only centrodes with curvature $\kappa(\ph) \le 0$  for $0 \le \ph \le 2\pi$ come into question.
In order to avoid gears with exotic shapes, only centrodes without self-intersections should be used.
A small investigation shows that $2 - \sqrt{2} \approx 0.585786$ is the biggest value of $b$ with $\kappa(\ph) \le 0$ for $0 \le \ph \le 2\pi$, and $\max_{\,0 \le \ph \le 2\pi} \kappa(\ph) = 0 = \kappa(0) = \kappa(2\pi)$.
For the centrode $\varXi_P$ of the driven gear one finds $\tilde\kappa(\ph) > 0$ for $b = 2 -\sqrt{2}$ and $0 \le \ph \le 2\pi$.
So in the following we use $b = 2 - \sqrt{2}$.
The graph of $\psi(\ph)$ is shown in Fig.\ \ref{Fig:Transmission_function01}.
Our transmission function is -- up to a phase difference of $\pi$ -- equal to the transmission function shown in \cite[Fig.\ 18\,b]{Litvin&Gonzalez-Perez&Fuentes&Hayasaka08}.

A numerical evaluation yields $I(0,2\pi) \approx 3.09315$ ($I(\ph_0,\ph_1)$ see \eqref{Eq:I}).
Choosing $m = 2$ and $z_1 = 14$, from \eqref{Eq:a} we get $a \approx 28.4385$.   
The centrodes $X_P$ and $\varXi_P$ with this value of $a$ are to be seen in Fig.\ \ref{Fig:Unrundraeder01} (cp.\ the centrodes in \cite[Fig.\ 18\,a]{Litvin&Gonzalez-Perez&Fuentes&Hayasaka08}).

\begin{figure}[h]
\begin{minipage}[c]{0.48\textwidth}
  \centering
  \includegraphics[scale=0.75]{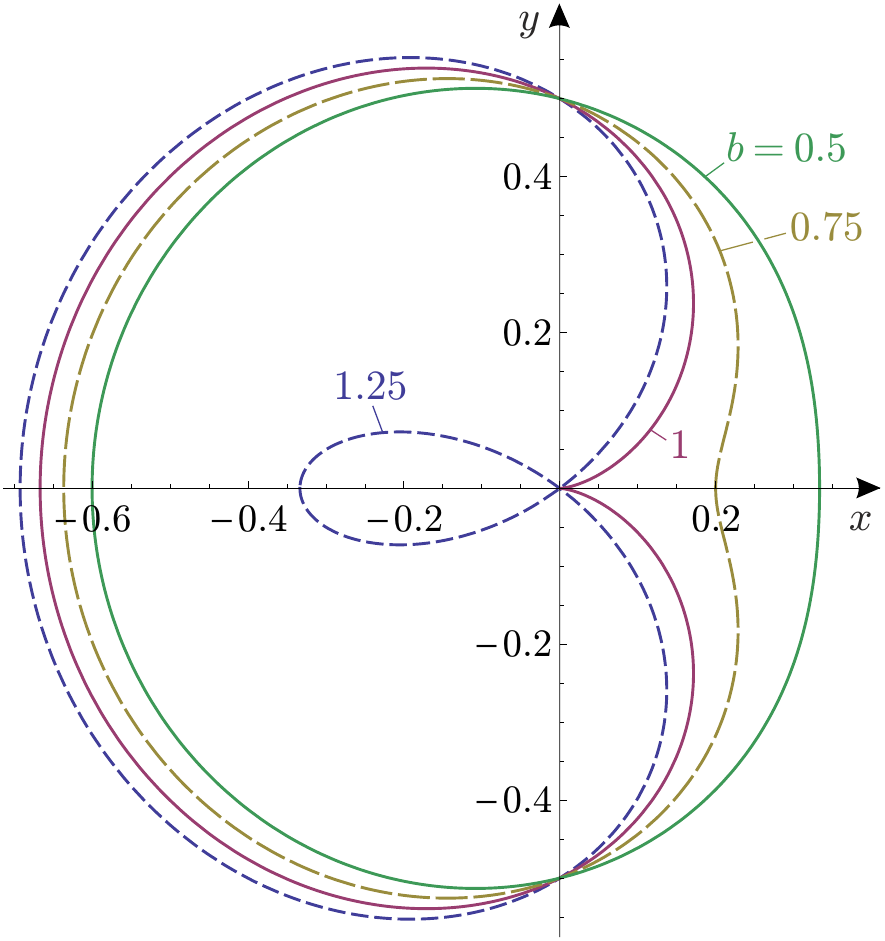}
  \caption{The centrode $X_P$ with $a=1$ and differerent values of the parameter $b$}
  \label{Fig:Centrodes01}
\end{minipage}
\hfill
\begin{minipage}[c]{0.48\textwidth}
  \centering
  \includegraphics[scale=0.75]{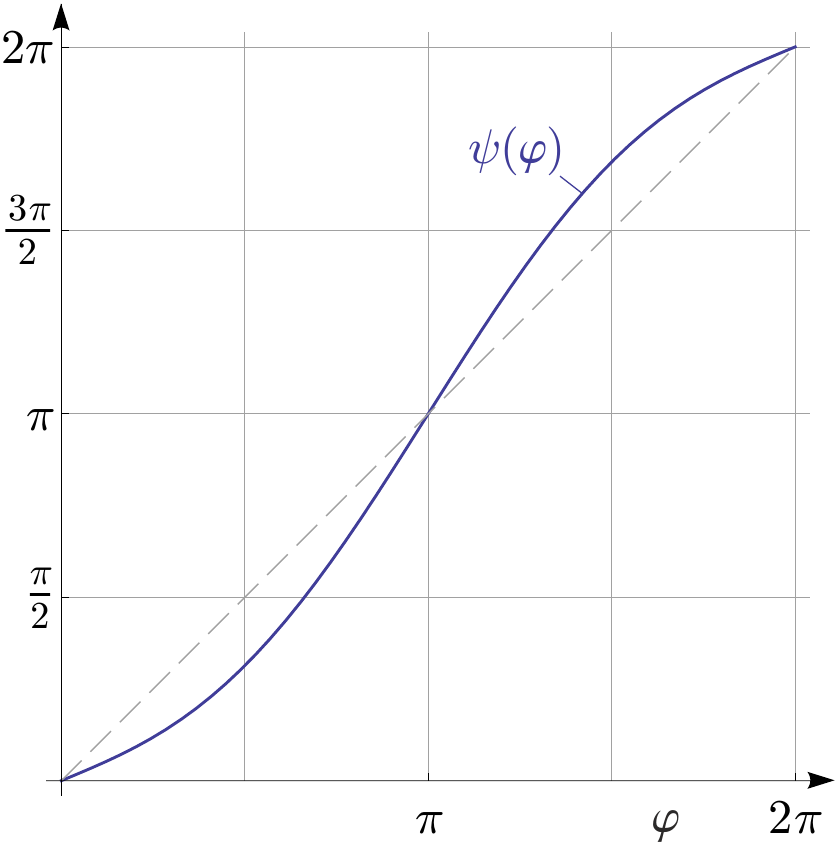}
  \caption{Graph of $\psi(\ph)$ with $b = 2 - \sqrt{2}$}
  \label{Fig:Transmission_function01}
\end{minipage}
\end{figure}

From \eqref{Eq:chi(k)} we get the values $\chi(1),\ldots,\chi(14)$ in Table \ref{Tab:chi(k)}. 

\begin{table}[h]
\caption{Values of $\chi(k)$}
\label{Tab:chi(k)}
\begin{center}
\renewcommand{\arraystretch}{1.2}
\begin{tabular}{|c|ccccccc|} \hline
$k$       & 1     & 2        & 3       & 4       & 5       & 6       & 7      \\
$\chi(k)$ & 0     & 0.674065 & 1.18877 & 1.63010 & 2.03297 & 2.41317 & 2.78037\\ \hline\hline
$k$       & 8     & 9        & 10      & 11      & 12      & 13      & 14     \\
$\chi(k)$ & $\pi$ & 3.50282  & 3.87002 & 4.25022 & 4.65309 & 5.09441 & 5.60912\\ \hline 
\end{tabular}
\renewcommand{\arraystretch}{1}
\end{center}
\end{table} 

After choosing $h_a = m = 2$, $\hf = 1.2\cdot m = 2.4$ and $\rh = 0.3\cdot m = 0.6$ ($h_a$, $\hf$, $\rh$ see Fig.\ \ref{Fig:Reference_profile01}), we get the pair of gears in Fig.\ \ref{Fig:Unrundraeder01}. 

\begin{figure}[h]
  \begin{center}
	\includegraphics[scale=1.2]{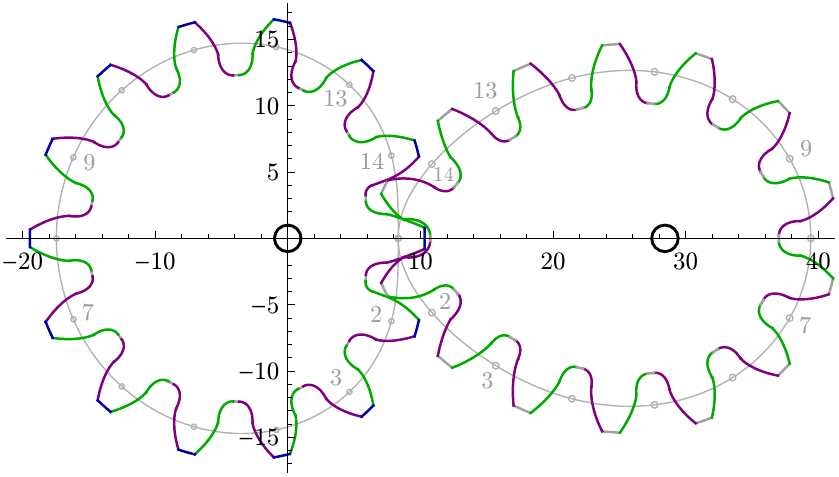}
	\caption{Example with $\psi(\ph) = \ph - \left(2-\sqrt{2}\,\right)\sin\ph$ and $z_1 = 14 = z_2$}
	\label{Fig:Unrundraeder01}
  \end{center}
\end{figure}

Table \ref{Tab:Undercut} shows the values $\phSkpm$ of $\ph$ at the singular points of $\XFkpm$ and the curvatures $\kappa(\phSkpm)$ of $X_P$ at $\phSkpm$ for the drive gear.
From \eqref{Eq:-kappa(ph_(S,pm))} we get $-\kappa(\phSkpm) \le 0.0583369$ as condition for non-undercutting.
So one sees that only the ``$-$''-flank of tooth 2, and the ``$+$''-flank of tooth 14 are free of undercut.
(The curves of tooth 14 are shown in Fig.\ \ref{Fig:Curves_for_tooth_k}.)

\begin{table}[h]
\caption{Drive gear: Singular points $\phSkpm$, curvatures $\kappa(\phSkpm)$ and undercut (UC)}
\label{Tab:Undercut}
\begin{center}
\renewcommand{\arraystretch}{1.2}
\begin{tabular}{|c|ccc|ccc|} \hline
$k$ & $\phSkm$    & $\kappa(\phSkm)$ & UC & $\phSkp$ & $\kappa(\phSkp)$ & UC \\ \hline
1   & $-0.662309$ & $-0.0793920$ & \tb & 0.662309 & $-0.0793920$ & \tb\\
2   & $-0.370208$ & $-0.0459235$ & $-$ & 1.13773  & $-0.0902213$ & \tb\\
3   & 0.697105    & $-0.0816165$ & \tb & 1.60927  & $-0.0827132$ & \tb\\
4   & 1.23593     & $-0.0892922$ & \tb & 2.04386  & $-0.0744589$ & \tb\\
5   & 1.64804     & $-0.0819279$ & \tb & 2.44631  & $-0.0690331$ & \tb\\
6   & 2.02344     & $-0.0748006$ & \tb & 2.82554  & $-0.0662394$ & \tb\\
7   & 2.38300     & $-0.0697165$ & \tb & 3.18979  & $-0.0655454$ & \tb\\
8   & 2.73727     & $-0.0666956$ & \tb & 3.54591  & $-0.0666956$ & \tb\\
9   & 3.09339     & $-0.0655454$ & \tb & 3.90019  & $-0.0697165$ & \tb\\
10  & 3.45764     & $-0.0662394$ & \tb & 4.25974  & $-0.0748006$ & \tb\\
11  & 3.83688     & $-0.0690331$ & \tb & 4.63514  & $-0.0819279$ & \tb\\
12  & 4.23933     & $-0.0744589$ & \tb & 5.04725  & $-0.0892922$ & \tb\\
13  & 4.67392     & $-0.0827132$ & \tb & 5.58608  & $-0.0816165$ & \tb\\
14  & 5.14546     & $-0.0902213$ & \tb & 6.65339  & $-0.0459235$ & $-$\\ \hline
\end{tabular}
\renewcommand{\arraystretch}{1}
\end{center}
\end{table}

Fig.\ \ref{Fig:Evolute_and_involutes01} shows for the drive gear the centrode $X_P$, the base curve (evolute) $\XBp$, the involutes $\XFkp$, $k = 1,\ldots,14$, and for comparison the involute $X_{F\!,\:\!14,\:\!-}$.
The base curve has one cusp and one point at infinity.
The point at infinity occurs for $\ph = 0$ and is caused by the fact that the straight lines whose envelope is the base curve have no intersection in the finite for $\ph = 0$. 
Note that the base curve $\XBm$ and all involutes $\XFkm$ are obtained by mirroring the corresponding ``$+$''-curves on the abscissa.
For example, $X_{F\!,\:\!14,\:\!-}$ corresponds to $X_{F\!,\:\!2,\:\!+}$. 

\begin{figure}[h]
  \begin{center}
	\includegraphics[scale=0.9]{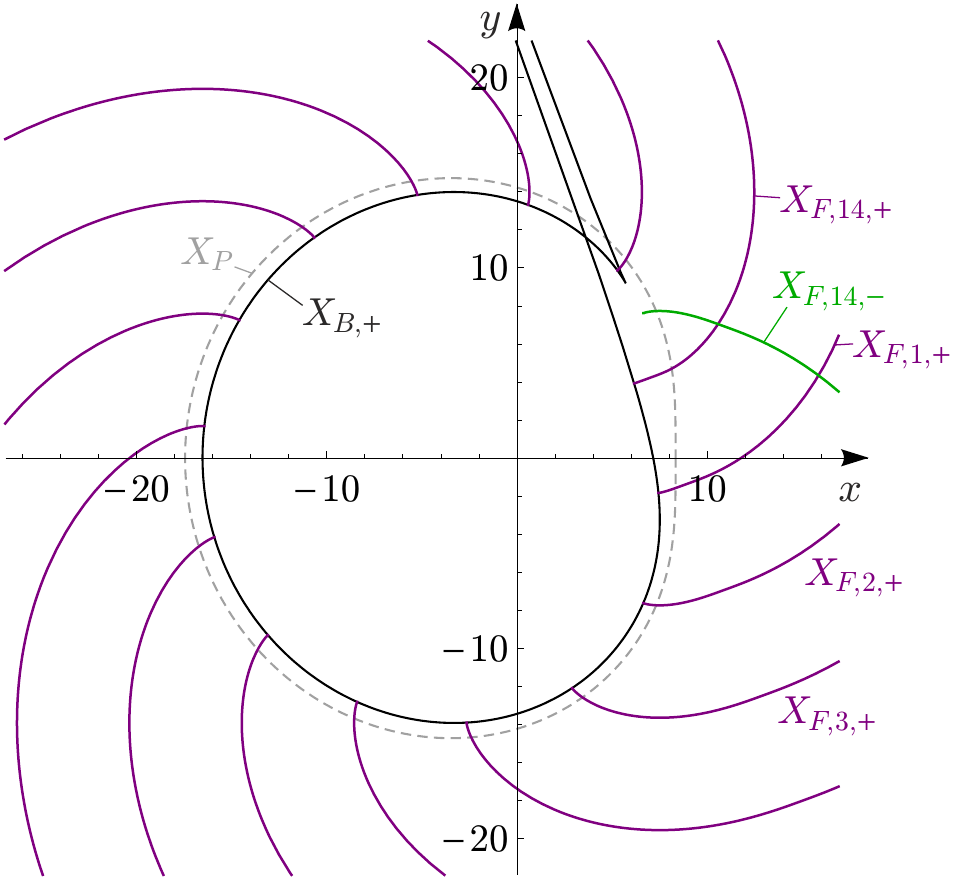}
	\caption{Centrode $X_P$ (dashed), base curve $\XBp$ and involutes of the drive gear}
	\label{Fig:Evolute_and_involutes01} 
  \end{center}
\end{figure}

\clearpage


\appendix

\section{Appendix: The exterior product}
\label{App:Exterior_product}

The {\em exterior product} (see \textcite[p.\ 414]{Luck&Modler}) of two complex numbers $A = x_A + \ii y_A$ and $B  = x_B + \ii y_B$ is a real number, and is defined by
\beqn \label{Eq:[A,B]-Def}
  [A,B]
= \frac{\ii}{2} \left(A\,\overline{B} - \overline{A}\,B\right)
= x_A y_B - y_A x_B\,.  
\eeqn
The real number $[A,B]$ is the signed area of the parallelogram spanned by the complex vectors $A$ and $B$.   
The exterior product is
\beqn \label{Eq:Rules1}
 \left.
  \begin{array}{ll}
	\mbox{anticommutative:}				& [A,B] = -[B,A]\,,\\[0.1cm]
	\mbox{distributive over addition:}	& [A+B,C] = [A,C] + [B,C]\,,\\[0.1cm]
	\mbox{\underline{not} associative:}	& \big[[A,B],C\big] \ne \big[A,[B,C]\big]\,.  
  \end{array}
 \right\} 
\eeqn
Further rules:
\begin{align}
& [1,A] = \Imz A\,,\label{Eq:Rule1}\\[0.1cm]
& [A,Bc] = [Ac,B] = c\,[A,B]\,,\quad c\in\R\,,\label{Eq:Rule2}\\[0.1cm]
& [A\,C,B\,C] = \left[A\,C\,\overline{C},B\right] = C\,\overline{C}\,[A,B]\,.\label{Eq:Rule3}
\end{align}
If $A$ and $B$ are functions of a real parameter $\ph$, then
\begin{align} \label{Eq:[A,B]'}
  [A(\ph),B(\ph)]'
\equiv {} & \frac{\dd}{\dd\ph}\,[A(\ph),B(\ph)]
= \frac{\ii}{2}\, \big(A'\,\overline{B} + A\,\overline{B}'
- \overline{A}'B - \overline{A}\,B'\big)\nonumber\\
= {} & \frac{\ii}{2}\, \big(A'\,\overline{B} - \overline{A}'\,B\big)
+ \frac{\ii}{2}\, \big(A\,\overline{B}' - \overline{A}\,B'\big)\nonumber\\[0.1cm]
= {} & [A'(\ph),B(\ph)] + [A(\ph),B'(\ph)]\qquad \mbox{(product rule)}\,. 
\end{align} 
Clearly, three points $A$, $B$, $X$ lie on a straight line $g$ if
\beq
  [A-X,B-X] = 0\,.
\eeq
From this equation we get
\begin{align*}
  0
= {} & [A,B-X] - [X,B-X]
= [A,B] - [A,X] - [X,B] + \underbrace{[X,X]}_0\\
= {} & [A,B] - \left([A,X] - [B,X]\right)
= [A,B] - [A-B,X]\,,  
\end{align*}
hence
\beqn \label{Eq:EOL} 
  [A-B,X] = [A,B]\,.
\eeqn
If $X$ is an arbitrary point of $g$, Eq.\ \eqref{Eq:EOL} can be considered as the equation of the line through the points $A$ and $B$.

\section{Appendix: Formula symbols}
\label{App:Formula_symbols}

\begin{longtable}{rp{13cm}} \toprule
$a$ & center distance $\rightarrow$ Fig.\ \ref{Fig:Waelzkurven01}, \eqref{Eq:a}\\ \midrule
$B$ & contact point of $\XFkpm$ and $\Xrkpm$ $\rightarrow$ \eqref{Eq:-+lambda*cos(alpha)}\\ \midrule
$\C$ & set of complex numbers $\rightarrow$ \eqref{Eq:X_P}, \eqref{Eq:Xi_P}\\
$C_\pm$ & point of the rack-cutter $\rightarrow$ Figures \ref{Fig:Reference_profile01} and \ref{Fig:Contact_point01}\\ \midrule
$D_\pm$ & point of the rack-cutter $\rightarrow$ Figures \ref{Fig:Reference_profile01} and \ref{Fig:Contact_point02}\\ \midrule
$h(\ph)$ & auxiliary function $\rightarrow$ \eqref{Eq:h}, \eqref{Eq:h=s'*kappa}\\
$\hXi(\ph)$ & auxiliary function $\rightarrow$ \eqref{Eq:hXi}\\
$h_a$ & addendum of the gear teeth, dedendum of the rack-cutter teeth $\rightarrow$ Fig.\ \ref{Fig:Reference_profile01}\\ 
$\hf$ & dedendum of the gear teeth, addendum of the rack-cutter teeth $\rightarrow$ Fig.\ \ref{Fig:Reference_profile01}\\
$\hf^*$ & factor $\rightarrow$ \eqref{Eq:hf^*_and_rho^*}\\ \midrule
$I(\ph_0,\ph_1)$ & integral $\rightarrow$ \eqref{Eq:I}\\ 
$\ii$ & imaginary unit, $\ii^2 = -1$\\ \midrule
$\ell_1$ & distance $\rightarrow$ Figures \ref{Fig:Reference_profile01} and \ref{Fig:Contact_point01}\\
$\ell_2$ & distance $\rightarrow$ Figures \ref{Fig:Reference_profile01} and \ref{Fig:Contact_point01}\\ 
$\ell_3$ & distance $\rightarrow$ Figures \ref{Fig:Reference_profile01} and \ref{Fig:Contact_point02}\\
$\ell_4$ & distance $\rightarrow$ Figures \ref{Fig:Reference_profile01} and \ref{Fig:Contact_point02}\\ \midrule
$M_\pm$ & mid point of the fillet of the rack-cutter ``$\pm$''-flank $\rightarrow$ Fig.\ \ref{Fig:Reference_profile01}\\
$m$ & module $\rightarrow$ Fig.\ \ref{Fig:Reference_profile01}, \eqref{Eq:a}\\ \midrule
$\NMkpm(\ph)$ & normal unit vector at point $\XMkpm(\ph)$ of the mid point curve $\XMkpm$ $\rightarrow$ \eqref{Eq:NMkpm}\\
$\NXiMkpm(\ph)$ & normal unit vector at point $\XiMkpm(\ph)$ of the mid point curve $\XiMkpm$ $\rightarrow$ \eqref{Eq:NXi_M}\\ \midrule
$O_1$ & pivot point of the drive gear $\rightarrow$ Fig.\ \ref{Fig:Waelzkurven01}\\
$O_2$ & pivot point of the driven gear $\rightarrow$ Fig.\ \ref{Fig:Waelzkurven01}\\ \midrule
$P$ & pitch point, instanteneous center of relative rotation of drive and driven gear $\rightarrow$ Fig.\ \ref{Fig:Waelzkurven01}\\
$p$ & tooth pitch $p = \pi m$ $\rightarrow$ Fig.\ \ref{Fig:Reference_profile01}\\ \midrule
$\R$ & set of real numbers\\
$R(\ph)$ & distance between $O_2$ and $P$ $\rightarrow$ Fig.\ \ref{Fig:Waelzkurven01}, \eqref{Eq:r_and_R}\\
$r(\ph)$ & distance between $O_1$ and $P$ $\rightarrow$ Fig.\ \ref{Fig:Waelzkurven01}, \eqref{Eq:r_and_R}\\
$R_1$ & point of the rack-cutter that coincides with $X_P(\chi(k))$ for $\ph = \chi(k)$ when the tooth $k$ of the driving gear is generated $\rightarrow$ Fig.\ \ref{Fig:Reference_profile01}\\
$R_2$ & point of the rack-cutter that coincides with $\varXi_P(\chi(k))$ for $\ph = \chi(k)$ when the tooth space $k$ of the driven gear is generated $\rightarrow$ Fig.\ \ref{Fig:Reference_profile01}\\ \midrule
$S$ & singularity of $\XFkpm$ $\rightarrow$ \eqref{Eq:lambda*kappa}\\
$s(\ph_0,\ph_1)$ & arc length between points $\ph_0$ and $\ph_1$ $\rightarrow$ \eqref{Eq:s}\\ \midrule
$t$ & time $\rightarrow$ \eqref{Eq:psi'(phi(t))}\\
$T(\ph)$ & tangent unit vector at point $X_P(\ph)$ of $X_P$ $\rightarrow$ \eqref{Eq:T}\\
$t_x$ & real part of $T(\ph)$ $\rightarrow$ \eqref{Eq:tx_and_ty}\\  
$t_y$ & imaginary part of $T(\ph)$ $\rightarrow$ \eqref{Eq:tx_and_ty}\\
$\TXi(\ph)$ & tangent unit vector at point $\varXi(\ph)$ of $\varXi_P$  $\rightarrow$ \eqref{Eq:TXi}\\
$t_\xi$ & real part of $\TXi(\ph)$ $\rightarrow$ \eqref{Eq:txi_and_teta}\\  
$t_\eta$ & imaginary part of $\TXi(\ph)$ $\rightarrow$ \eqref{Eq:txi_and_teta}\\ \midrule
$w(\ph)$ & auxiliary function $\rightarrow$ \eqref{Eq:w}\\ \midrule
$X_a$ & addendum curve of the drive gear: $X_a := X_{p,h_a,+}$ $\rightarrow$ Fig.\ \ref{Fig:Curves_for_tooth_k}, \eqref{Eq:Xpdpm}\\
$\XBpm(\ph)$ & parametric equation of the bases curves (evolutes) $\XBpm$ for the flank curves (involutes) ``$+$'' and ``$-$'' of the drive gear $\rightarrow$ \eqref{Eq:X_B}\\  
$\XFkpm(\ph)$ & parametric equation of the flank curves $\XFkpm$ of the tooth $k$ (see Fig.\ \ref{Fig:Curves_for_tooth_k}) of the drive gear $\rightarrow$ \eqref{Eq:XFkpm}\\
$\xFkpm(\ph)$ & real part of $\XFkpm(\ph)$ $\rightarrow$ Corollary \ref{Coro:xF_and_yF}\\
$\XFpm^*(\ph)$ & parametric equation of the involutes $\XFpm^*$ of the base curve $\XBpm$ $\rightarrow$ Theorem~\ref{Thm:XFpm^*}\\
$X_f$ & dedendum curve of the drive gear: $X_f := X_{p,\hf,-}$ $\rightarrow$ Fig.\ \ref{Fig:Curves_for_tooth_k}, \eqref{Eq:Xpdpm}\\
$\XMkpm(\ph)$ & parametric equations of the curves $\XMkp$ and $\XMkm$ (see Fig.\ \ref{Fig:Curves_for_tooth_k}) of the points $M_+$ and $M_-$, respectively, in Fig.\ \ref{Fig:Reference_profile01} $\rightarrow$ \eqref{Eq:XMkpm}\\ 
$X_P(\ph)$ & parametric equation of the centrode $X_P$ of the drive gear $\rightarrow$ \eqref{Eq:X_P}, Fig.\ \ref{Fig:Waelzkurven01}\\
$\Xpdpm(\ph)$ & parametric equations of an outer (``$+$'') and inner (``$-$'') parallel curve of $X_P$ with distance $d$ $\rightarrow$~\eqref{Eq:Xpdpm}\\
$\XWkpm(\ph,\mu)$ & parametric equations of the rack-cutter flank lines for the generation of tooth $k$ of the drive gear $\rightarrow$ \eqref{Eq:XWkpm} (for the definition of $``+$'' and  $``-$'' see Fig.\ \ref{Fig:Reference_profile01})\\
$\xWkpm(\ph,\mu)$ & real part $\XWkpm(\ph,\mu)$ $\rightarrow$ \eqref{Eq:xWk_and_yWk}\\  
$\Xrkpm(\ph)$ & parametric equations of the parallel curves (fillets) $\Xrkpm$ (see Fig.\ \ref{Fig:Curves_for_tooth_k}) of $\XMkpm$ with distance $\rh$ $\rightarrow$ \eqref{Eq:Xrkpm}\\ \midrule
$\yFkpm(\ph)$ & imaginary part of $\XFkpm(\ph)$ $\rightarrow$ Corollary \ref{Coro:xF_and_yF}\\
$\yWkpm(\ph,\mu)$ & imaginary part of $\XWkpm(\ph,\mu)$ $\rightarrow$ \eqref{Eq:xWk_and_yWk}\\ \midrule
$z_1$ & number of teeth of the drive gear $\rightarrow$ \eqref{Eq:a}\\
$z_2$ & number of teeth of the driven gear\\ \midrule
$\alpha$ & rack-cutter profile angle $\rightarrow$ Fig.\ \ref{Fig:Reference_profile01}\\ \midrule
$\gamma$ & rotation angle of the driven gear  $\rightarrow$ Fig.\ \ref{Fig:Waelzkurven01}, \eqref{Eq:gamma}\\ \midrule
$\etaFkpm(\ph)$ & imaginary part of $\XiFkpm(\ph)$ $\rightarrow$ Corollary \ref{Coro:xiF_and_etaF}\\
$\etaWkpm(\ph,\mu)$ & imaginary part $\XiWkpm(\ph,\mu)$ $\rightarrow$ \eqref{Eq:xiWk_and_etaWk}\\ \midrule
$\kappa(\ph)$ & curvature at point $\ph$ of $X_P$ $\rightarrow$ \eqref{Eq:kappa(phi)}\\
$\kaXi(\ph)$ & curvature at point $\ph$ of $\varXi_P$ $\rightarrow$ \eqref{Eq:kaXi(phi)}\\ \midrule
$\lak(\ph)$ & auxiliary function $\rightarrow$ \eqref{Eq:lambda_k}\\ \midrule
$\varXi_a$ & addendum curve of the driven gear: $\varXi_a := \varXi_{p,h_a,+}$, \eqref{Eq:Xipdpm}\\
$\XiBpm(\ph)$ & parametric equation of the base curves (evolutes) $\XiBpm$ for the flank curves (involutes) ``$+$'' and ``$-$'' of the driven gear $\rightarrow$ \eqref{Eq:Xi_B}\\
$\XiFkpm(\ph)$ & parametric equations of the flank curves $\XiFkpm$ of the tooth space $k$ of the driven gear $\rightarrow$ \eqref{Eq:XiFkpm}\\
$\xiFkpm(\ph)$ & real part of $\XiFkpm(\ph)$ $\rightarrow$ Corollary \ref{Coro:xiF_and_etaF}\\  
$\XiFpm^*(\ph)$ & parametric equation of the involutes $\XiFpm^*$ of the base curve $\XiBpm$ $\rightarrow$ Theorem~\ref{Thm:XiFpm^*}\\
$\varXi_f$ & dedendum curve of the driven gear: $\varXi_f := \varXi_{p,\hf,-}$, \eqref{Eq:Xipdpm}\\
$\XiMkpm(\ph)$ & parametric equation of the curves $\XiMkpm$ of the points $M_+$ and $M_-$, respectively, in Fig.\ \ref{Fig:Reference_profile01} $\rightarrow$ \eqref{Eq:XiMkpm}\\
$\varXi_P(\ph)$ & parametric equation of the centrode $\varXi_P$ of the driven gear $\rightarrow$ \eqref{Eq:Xi_P}, Fig.\ \ref{Fig:Waelzkurven01}\\
$\Xipdpm(\ph)$ & parametric equation of an outer (``$+$'') and inner (``$-$'') parallel curve of $\varXi_P$ with distance $d$ $\rightarrow$~\eqref{Eq:Xipdpm}\\
$\XiWkpm(\ph,\mu)$ & parametric equation of the rack-cutter flanks for the generation of tooth space~$k$ of the driven gear $\rightarrow$ \eqref{Eq:XiWkpm} (for the definition of $``+$'' and  $``-$'' see Fig.\ \ref{Fig:Reference_profile01})\\
$\xiWkpm(\ph,\mu)$ & real part of $\XiWkpm(\ph,\mu)$ $\rightarrow$ \eqref{Eq:xiWk_and_etaWk}\\
$\Xirkpm(\ph)$ & parametric equation of the parallel curves (fillets) $\Xirkpm$ with distance $\rh$ to curve $\XiMkpm$ $\rightarrow$ \eqref{Eq:Xirkpm}\\ \midrule
$\rh$ & fillet radius $\rightarrow$ Fig.\ \ref{Fig:Reference_profile01}\\
$\rh^*$ & factor $\rightarrow$ \eqref{Eq:hf^*_and_rho^*}\\ \midrule
$\ph$ & rotation angle of the drive gear $\rightarrow$ Fig.\ \ref{Fig:Waelzkurven01}\\
$\ph^*(t)$ & rotation angle of the drive gear as time function $\rightarrow$ \eqref{Eq:psi^*}\\
$\dot\ph^*(t)$ & angular velocity of the drive gear $\rightarrow$ \eqref{Eq:psi'(phi(t))}\\
$\phAkpE$ & value of $\ph$ with $\Xrkp(\ph) \cap X_f \ne \emptyset$ $\rightarrow$ \eqref{Eq:+lambda_k}, \eqref{Eq:+-flank}, \eqref{Eq:arc_dedendum_curve_a}, \eqref{Eq:arc_dedendum_curve_b}\\
$\phAkpZ$ & value of $\ph$ with $\Xrkp(\ph)\cap\XFkp\ne\emptyset$ $\rightarrow$ \eqref{Eq:XFkp=Xrkp} or \eqref{Eq:-+lambda*cos(alpha)}, \eqref{Eq:+-flank}\\
$\phFkpE$ & value of $\ph$ with $\XFkp(\ph) \cap X_a \ne \emptyset$ $\rightarrow$ \eqref{Eq:XFkp=Xha}, \eqref{Eq:+-flank}\\
$\phFkpZ$ & value of $\ph$ with $\XFkp(\ph)\cap\Xrkp\ne\emptyset$ $\rightarrow$ \eqref{Eq:XFkp=Xrkp} or \eqref{Eq:-+lambda*cos(alpha)}, \eqref{Eq:+-flank}\\
$\phpkpE$ & value of $\ph$ with $X_a(\ph)\cap\XFkm\ne\emptyset$ $\rightarrow$ \eqref{Eq:XFkm=Xha}, \eqref{Eq:arc_addendum_curve}\\
$\phpkpZ$ & value of $\ph$ with $X_a(\ph)\cap\XFkp\ne\emptyset$ $\rightarrow$ \eqref{Eq:XFkp=Xha}, \eqref{Eq:arc_addendum_curve}\\
$\phAkmE$ & value of $\ph$ with $\Xrkm(\ph)\cap\XFkm\ne\emptyset$ $\rightarrow$ \eqref{Eq:XFkm=Xrkm} or \eqref{Eq:-+lambda*cos(alpha)}, \eqref{Eq:--flank}\\
$\phAkmZ$ & value of $\ph$ with $\Xrkm(\ph) \cap X_f \ne \emptyset$ $\rightarrow$ \eqref{Eq:--flank}, \eqref{Eq:arc_dedendum_curve_a}, \eqref{Eq:arc_dedendum_curve_b}\\
$\phFkmE$ & value of $\ph$ with $\XFkm(\ph)\cap\Xrkm\ne\emptyset$ $\rightarrow$ \eqref{Eq:XFkm=Xrkm} or \eqref{Eq:-+lambda*cos(alpha)}, \eqref{Eq:--flank}\\
$\phFkmZ$ & value of $\ph$ with $\XFkm(\ph) \cap X_a\ne\emptyset$ $\rightarrow$ \eqref{Eq:XFkm=Xha}, \eqref{Eq:--flank}\\
$\phSkpm$ & value of $\ph$ at the singular point of $\XFkpm$ $\rightarrow$ \eqref{Eq:XFkpm'(phi)=0}, \eqref{Eq:lambda*kappa}\\
$\phSpm$ & value of $\ph$ at the singular point of $\XFpm^*$ ($\rightarrow$ \eqref{Eq:XFp^*} and \eqref{Eq:XFm^*}) or $\XiFpm^*$ ($\rightarrow$~Theorem \ref{Thm:XiFpm^*})\\
$\phXiAkpE$ & value of $\ph$ with $\Xirkp(\ph)\cap\XiFkp\ne\emptyset$ $\rightarrow$ \eqref{Eq:XiFkp=Xirkp} of \eqref{Eq:+-lambda*cos(alpha)}, \eqref{Eq:+-flankXi}\\
$\phXiAkpZ$ & value of $\ph$ with $\Xirkp(\ph) \cap \varXi_f \ne \emptyset$ $\rightarrow$ \eqref{Eq:+-lambda_k}, \eqref{Eq:+-flankXi}\\
$\phXiFkpE$ & value of $\ph$ with $\XiFkp(\ph)\cap\Xirkp\ne\emptyset$ $\rightarrow$  \eqref{Eq:XiFkp=Xirkp} or \eqref{Eq:+-lambda*cos(alpha)}, \eqref{Eq:+-flankXi}\\
$\phXiFkpZ$ & value of $\ph$ with $\XiFkp(\ph) \cap \varXi_a \ne \emptyset$ $\rightarrow$ \eqref{Eq:XiFkp=Xiha}, \eqref{Eq:+-flankXi}\\
$\phXipkpE$ & value of $\ph$ with $\varXi_a(\ph)\cap\XiFkm\ne\emptyset$ $\rightarrow$ \eqref{Eq:XiFkm=Xiha}, \eqref{Eq:arc_addendum_curve_a_Xi}, \eqref{Eq:arc_addendum_curve_b_Xi}\\
$\phXipkpZ$ & value of $\ph$ with $\varXi_a(\ph)\cap\XiFkp\ne\emptyset$ $\rightarrow$ \eqref{Eq:XiFkp=Xiha}, \eqref{Eq:arc_addendum_curve_a_Xi}, \eqref{Eq:arc_addendum_curve_b_Xi}\\
$\phXiAkmE$ & value of $\ph$ with $\Xirkm(\ph) \cap \varXi_f \ne \emptyset$ $\rightarrow$ \eqref{Eq:+-lambda_k}, \eqref{Eq:--flankXi}, \eqref{Eq:arc_dedendum_curve_Xi}\\
$\phXiAkmZ$ & value of $\ph$ with $\Xirkm(\ph)\cap\XiFkm\ne\emptyset$ $\rightarrow$ \eqref{Eq:XiFkm=Xirkm} or \eqref{Eq:+-lambda*cos(alpha)}, \eqref{Eq:--flankXi}\\
$\phXiFkmE$ & value of $\ph$ with $\XiFkm(\ph) \cap \varXi_a \ne \emptyset$ $\rightarrow$ \eqref{Eq:XiFkm=Xiha}, \eqref{Eq:--flankXi}\\
$\phXiFkmZ$ & value of $\ph$ with $\XiFkm(\ph)\cap\Xirkm\ne\emptyset$ $\rightarrow$ \eqref{Eq:XiFkm=Xirkm} or \eqref{Eq:+-lambda*cos(alpha)}, \eqref{Eq:--flankXi}\\
$\phXiSkpm$ & value of $\ph$ at the singular point of $\XiFkpm$ $\rightarrow$ \eqref{Eq:XiFkpm'(phi)=0}, \eqref{Eq:lambda*kaXi}\\
\midrule
$\chi(k)$ & value of $\ph$ in the mid of tooth $k$, $k = 1,2,\ldots,z_1$ of the drive gear $\rightarrow$ \eqref{Eq:chi(k)}, value of $\ph$ in the mid of tooth space $k$, $k = 1,2,\ldots,z_2$ of the driven gear\\ \midrule
$\psi(\ph)$ & transmission function $\rightarrow$ \eqref{Eq:psi}, \eqref{Eq:gamma}\\
$\psi^*(t)$ & rotation angle of the driven gear as time function $\rightarrow$ \eqref{Eq:psi^*}\\
$\dot\psi^*(t)$ & angular velocity of the driven gear $\rightarrow$ \eqref{Eq:psi'(phi(t))}\\ \midrule
$\overline A$ & conjugate $a - b\ii$ of the complex number $A = a + b\ii$\\
$[A,B]$ & exterior product of the complex numbers $A$ and $B$ $\rightarrow$ Appendix \ref{App:Exterior_product}\\
\bottomrule
\end{longtable}

\printbibliography

\vspace{0.2cm}
English translations of German titles:
\begin{itemize}
\setlength{\itemsep}{-2pt}
\item[\cite{Baule1}\;\:] {\em Differential und integral calculus}
\item[\cite{DIN3972}\;\:] {\em Reference profiles for rack-cutters for involute gears according to DIN 867}
\item[\cite{Kuehnel}\;\:] {\em Differential geometry. Curves - surfaces - manifolds}
\item[\cite{Luck&Modler}\;\:] {\em Design of mechanisms -- Analysis, synthesis, optimization}
\item[\cite{Mueller:Kinematik}\;\:] {\em Kinematics}
\item[\cite{Wunderlich}\;\:] {\em Plane kinematics}
\end{itemize}

\vspace{1cm}

\begin{center}
Uwe B{\"a}sel\\[0.2cm]
HTWK Leipzig, University of Applied Sciences,\\
Faculty of Engineering,\\
PF 30 11 66, 04251 Leipzig, Germany\\[0.2cm]
\texttt{uwe.baesel@htwk-leipzig.de} 
\end{center}

\end{document}